\documentclass[10pt]{amsart}
\usepackage{amssymb}
\usepackage[english]{babel}
\usepackage{fancyhdr}
\usepackage{colortbl}
\usepackage{enumerate}
\usepackage{graphicx} 
\usepackage{float}
\usepackage{mathrsfs}
\usepackage{cite}
\usepackage{kpfonts}
\usepackage[left=2.3cm,right=2.3cm,top=2.65cm,bottom=2.65cm]{geometry}
\usepackage{tikz}
\usetikzlibrary{arrows,automata,shapes,calc,patterns}

\allowdisplaybreaks

\newtheorem{thm}{Theorem}[section]
\newtheorem{prop}[thm]{Proposition}
\newtheorem{lem}[thm]{Lemma}
\newtheorem{cor}[thm]{Corollary}

\theoremstyle{definition}

\theoremstyle{remark}
\newtheorem{remark}{Remark}[section]

\numberwithin{equation}{section}

\theoremstyle{notation}

\usepackage[colorlinks=true,urlcolor=green!50!black,
citecolor=red!60!black,linkcolor=blue!75!black,linktocpage,pdfpagelabels,
bookmarksnumbered,bookmarksopen]{hyperref}

\usepackage{amsmath}
\usepackage{amsthm}
\usepackage{latexsym}
\usepackage{amssymb}
\usepackage{mathtools}
\newcommand{\R}{\mathbb{R}}
\newcommand{\RN}{\mathbb{R}^N}
\newcommand{\RNp}{\mathbb{R}^{N}_{+}}
\newcommand{\N}{\mathbb{N}}

\renewcommand{\epsilon}{\varepsilon}

\newcommand{\gradu}{\nabla u}

\newcommand{\gradpsi}{ \nabla \psi}

\newcommand{\eps}{\varepsilon}

\newcommand{\cK}{{\mathcal K}}

\newcommand{\dist}{{\rm dist}}
\newcommand{\supp}{{\rm supp}}
\newcommand{\weakto}{\rightharpoonup}

\newcommand{\woo}{{c_p}}

\renewcommand{\epsilon}{\varepsilon}

\begin{document}

\title[The nonlinear Schr\"odinger equation in the half-space]{The nonlinear Schr\"odinger equation in the half-space}

\author[A. J. Fern\'andez and T. Weth]{Antonio J. Fern\'andez and Tobias Weth}

\address{
\newline
\textbf{{\small Antonio J. Fern\'andez}} 
\newline \indent Department of Mathematical Sciences, University of Bath, Bath BA2 7AY, UK.}
%\newline \indent Universit\'e Polytechnique Hauts-de-France, EA 4015-LAMAV-FR CNRS 2956, F-59313 Valenciennes, France \&
%\newline \indent Laboratoire de Math\'ematiques (UMR 6623), Universit\'e de Bourgogne Franche-Comt\'e,
%\newline \indent 16, Route de Gray 25030 Besan\c con Cedex, France}
\email{ajf77@bath.ac.uk} 

\address{
\vspace{-0.25cm}
\newline
\textbf{{\small Tobias Weth}}
\newline \indent Institut f\"ur Mathematik, Goethe-Universit\"at Frankfurt, Robert-Mayer-Str. 10, D-60629 Frankfurt am Main, Germany.}
\email{weth@math.uni-frankfurt.de}

\date{}
\subjclass[2010]{}

\maketitle
 
\begin{abstract} 
The present paper is concerned with the half-space Dirichlet problem
\vspace{0.2cm}
\begin{equation} \label{problem-abstract} \tag{$P_c$}
-\Delta v + v = |v|^{p-1}v,\ \textup{ in } \RNp, \qquad v = c,\ \textup{ on } \partial \RNp,\ \qquad \lim_{x_N \to \infty} v(x',x_N) = 0 \textup{ uniformly in }x' \in \R^{N-1},
\end{equation}
where $\RNp := \{\,x \in \RN: x_N > 0\, \}$ for some $N \geq 1$ and $p > 1$, $c > 0$ are constants. We analyse the existence, non-existence and multiplicity of bounded positive solutions to \eqref{problem-abstract}. We prove that the existence and multiplicity of bounded positive solutions to \eqref{problem-abstract} depend in a striking way on the value of $c > 0$ and also on the dimension $N$. We find an explicit number $\woo \in (1,\sqrt{e})$, depending only on $p$, which determines the threshold between existence and non-existence. In particular, in dimensions $N \geq 2$, we prove that, for $0 < c < \woo$, problem \eqref{problem-abstract} admits infinitely many bounded positive solutions, whereas, for $c > \woo$, there are no bounded positive solutions to \eqref{problem-abstract}.
\end{abstract}

\section{Introduction} \label{section-introduction}

\noindent Due to its relevance within several models arising in physics and biology, the nonlinear stationary Schr\"odinger equation
\begin{equation} \label{problem-RN-intro}
-\Delta v + v = |v|^{p-1} v, \quad \textup{ in } \RN,
\end{equation}
received extensive attention in the last four decades. In particular, let us mention that the study of solitary wave solutions for the (focusing) NLS
$$
i \partial_t \varphi + \Delta \varphi + |\varphi|^{p-1}\varphi = 0, \quad (t,x) \in \R \times \RN,
$$
is reduced to problem \eqref{problem-RN-intro} via a time-harmonic ansatz. For classic existence and multiplicity results, we refer the reader e.g. to the seminal papers \cite{BeLiI1983, BeLiII1983, BaWi1993, BaWi1993-2, BeGaKa1983, St1977, Ra1992} and the monographs \cite{Wi1996, St2008}. We also recall the fundamental works \cite{GiNiNi1981, Kw1989} where the radial symmetry and uniqueness, up to translations, of positive solutions to \eqref{problem-RN-intro} satisfying the decay condition
\begin{equation} \label{problem-RN-intro-decay-condition}
v(x) \to 0, \qquad \text{as $|x| \to \infty$},
\end{equation}
are proved in the case $1 < p < 2^*-1$. Here and in the following $2^*$ denotes the critical Sobolev exponent, i.e. $2^*= \frac{2N}{N-2}$ for $N \ge 3$ and $2^* =2^*-1 = \infty$ for $N=1,2$.
In particular, these results imply the uniqueness, up to translations, of positive finite energy solutions $u \in H^1(\R^N)$. In contrast, for $N \geq 2$,  \eqref{problem-RN-intro} admits an abundance of sign-changing finite energy solutions satisfying (\ref{problem-RN-intro-decay-condition}), see e.g. \cite{BaWi1993, BaWi1993-2,MuPaWe2012,lorca-ubilla-2004} and the references therein. Moreover, more recent geometric constructions of different solution shapes highlight the rich structure of the set of positive solutions which do not satisfy the decay assumption (\ref{problem-RN-intro-decay-condition}), see e.g. \cite{dancer-2001,AoMuPaWe2016,MuPaWe2012,liu-wei-2019} and the references therein.

Whereas it seems impossible to provide an exhaustive list of references for the full space problem (\ref{problem-RN-intro}), much less is known regarding the half-space Dirichlet problem  
\begin{equation} \label{mainProblem}
\left\{
\begin{aligned}
 -\Delta v + v & = |v|^{p-1} v, \quad & \textup{ in } \RNp,\\
 v & = c, & \quad \textup{ on } \partial \RNp,
\end{aligned}
\right.
\end{equation}
where $\RNp := \{ x \in \RN: x_N > 0\}$ for some $N \geq 1$ and $c \ge 0$ is a constant. In the case $c=0$, general nonexistence results are available for (\ref{mainProblem}). More precisely, the non-existence of finite energy solutions $u \in H_0^1(\RNp)$ to (\ref{mainProblem}) in the case $c=0$ follows from \cite[Theorem I.1]{EsLi1982}, while \cite[Corollary 1.3]{BeCaNi1997} yields, in particular, the non-existence of positive solutions to (\ref{mainProblem}) with $c=0$ and the the decay property
\begin{equation} \label{cond-infinity}
\lim_{x_N \to \infty} v(x',x_N) = 0, \qquad \text{uniformly in $x' \in \R^{N-1}$.}
\end{equation}

The aim of the present paper is to analyse the existence, non-existence and multiplicity of bounded positive solutions $v$ to the problem \eqref{mainProblem}-\eqref{cond-infinity} in the case $c>0$, for which we are not aware of any previous result in general dimensions $N$. As we shall see below, the multiplicity of positive solutions depends in a striking way on the value $c > 0$ and, somewhat surprisingly, also on the dimension $N$. Let us stress that we cannot expect the existence of finite energy solutions $u \in H^1(\RNp)$ to \eqref{mainProblem} in the case $N \ge 2$. Actually, we cannot expect solutions to \eqref{mainProblem} belonging to $L^p(\RNp)$ for any $1 \leq p < \infty$. The one-dimensional decay condition \eqref{cond-infinity} therefore seems natural. % More precisely, we are going to look for solutions $v$ to \eqref{mainProblem} satisfying 
%\begin{equation}
%v \in C^2(\RNp) \cap C(\overline{\RNp}) \cap L^{\infty}(\RNp) \quad \textup{ and } \quad \lim_{x_N \to \infty} v(x',x_N) = 0, \quad \textup{ uniformly in } x' \in \R^{N-1}.
%\end{equation}

The consideration of inhomogeneous Dirichlet boundary conditions as in \eqref{mainProblem} is to some extend motivated by recent works on pattern formation in biological and chemical models. For instance, in \cite{KrKlMaHeGa2020}, the authors numerically show that, for different types of reaction-diffusion systems, the pattern formation can be isolated away from the boundary using this type of boundary conditions. Moreover, within a rigorous analysis of some of these models in an asymptotically small diffusivity ratio, one may expect that the equation of the limiting profile is precisely \eqref{mainProblem} with $c > 0$. This is indeed the case for the Gierer-Meinhardt system considered in \cite{GoWe2020}. As one can observe in \cite[Section 2]{GoWe2020}, the \textit{homoclinic solution} $w_0 (\, \cdot+t_{c,p})$ (see \eqref{eq:def-w-0} and \eqref{eq:def-t-c-p} below) plays a key role in the construction of multi-spike patterns for this model.

Not surprisingly, problem~(\ref{mainProblem}) is completely understood in the case $N=1$. This is due to the fact that the one-dimensional equation 
\begin{equation} \label{1dim}
-w'' + w = w^p,
\end{equation}
admits a first integral, and from this one can easily deduce that, for $1<p< \infty$, (\ref{1dim}) admits, up to sign and translation, a unique global non trivial solution satisfying $w(t) \to 0$ as $t \to \pm \infty$. See e.g. \cite[Theorem 5]{BeLiI1983}, \cite[Theorem 3.16]{Wi1996} and  \cite[Theorem 1.2]{JeTa2003} where, in addition, the Mountain-Pass characterization of this unique solution is established. By direct computations, one can verify that this solution is precisely given by 
\begin{equation}
  \label{eq:def-w-0}
t \mapsto w_0(t) = \woo \left[ \cosh \left( \frac{p-1}{2} t \right) \right]^{-\frac{2}{p-1}}, \qquad \text{with}\qquad \woo: = \left( \frac{p+1}{2} \right)^{\frac{1}{p-1}} = w_0(0) = \sup_{t \in \R} w_0(t).
\end{equation}
As we shall see below, the value $\woo$ will be of key importance also for the higher dimensional version of (\ref{mainProblem}). The following complete characterization of the one dimensional case is an immediate consequence of these facts.

\begin{prop}
\label{prop-N=1}  
Let $N =1$, $p>1$ and $c>0$. Then:
\begin{itemize} 
\item[(i)] If $0 < c < \woo$, problem \eqref{mainProblem}-\eqref{cond-infinity} admits exactly two positive solutions given by $t \mapsto w_0(t + t_{c,p})$ and $t \mapsto w_0(t - t_{c,p})$, where 
  \begin{equation}
    \label{eq:def-t-c-p}
t_{c,p}:= \frac{2}{p-1} \ln \left( \sqrt{\frac{p+1}{2c^{p-1}}} + \sqrt{\frac{p+1}{2c^{p-1}}-1}\right).
\end{equation}
\item[(ii)] If $c = \woo$,  the function $w_0$ is the unique positive solution to \eqref{mainProblem}-\eqref{cond-infinity}.
\item[(iii)] If $c > \woo$, problem \eqref{mainProblem}-\eqref{cond-infinity} does not admit solutions.
\end{itemize}
\end{prop}

Our main results concern dimensions $N \ge 2$. In this case, problem (\ref{mainProblem})-(\ref{cond-infinity}) is invariant under translations and rotations parallel to the boundary $\partial \R^N_+ = \R^{N-1}$. In particular, if $N \ge 2$ and $v$ is a positive solution to (\ref{mainProblem})-(\ref{cond-infinity}), then the functions $x \mapsto v(x' + \tau,x_N)$, $\tau \in \R^{N-1}$ are also solutions to (\ref{mainProblem})-(\ref{cond-infinity}), where, here and in the following, we write $x = (x',x_N)$ for $x \in \R^N_{+}$ with $x' \in \R^{N-1}$. In the following, we call two solutions {\em geometrically distinct} if they do not belong to the same orbit of solutions under translations and rotations in $\R^{N-1}$.

\medbreak
Our first main result reads as follows.

\begin{thm} \label{th1}
Let $ N \geq 2$, $p > 1$ and $c > 0$. Then:
\begin{itemize}
\item[(i)] If $0 < c < \woo$, there exist at least three geometrically distinct bounded positive solutions to \eqref{mainProblem}-\eqref{cond-infinity}. 
\item[(ii)] If $c > \woo$, there are no bounded positive solutions to \eqref{mainProblem}-\eqref{cond-infinity}.
\end{itemize}
\end{thm}

\begin{remark} \label{remark-th1}$ $
\begin{itemize}
\item[(a)] By a bounded positive solution to \eqref{mainProblem}-\eqref{cond-infinity}, we mean a positive function $v \in C^2(\RNp) \cap C(\overline{\RNp}) \cap L^{\infty}(\RNp)$ satisfying \eqref{mainProblem} in the pointwise sense and such that \eqref{cond-infinity} holds. \smallbreak
\item[(b)] The nonexistence part (ii) of Theorem~\ref{th1} is proved with a variant of the sliding method based on a comparison with $x_N$-translates of the function $x \mapsto w_0(x_N)$. We shall comment on this in more detail further below.  
\smallbreak
\item[(c)] 
As we have indicated above already, Theorem \ref{th1} (i) highlights the rich structure of solutions to (\ref{mainProblem})-(\ref{cond-infinity}) in the case $N \ge 2$, $0< c< \woo$, which is in striking contrast to the case $N =1$. Indeed, Theorem \ref{th1} (i) yields, in addition to the two one-dimensional profile solutions $x \mapsto w_0(x_N + t_{c,p})$ and $x \mapsto w_0(x_N - t_{c,p})$ at least one further geometrically distinct solution which is not merely a function of the $x_N$-variable. Hence, by the remarks above, this one solution gives rise to an infinite number of positive solutions via translations parallel to the boundary $\partial \R^N_+ = \R^{N-1}$. In Corollary~\ref{th1-existence-refined-multiplicity}  below, we shall derive more precise lower estimates on the number of geometrically distinct solutions depending on the exponent $p$ and the dimension $N$.
\smallbreak
\item[(d)] It suffices to prove Theorem \ref{th1} (i) in the case $N=2$ since every positive solution $v$ to \eqref{mainProblem}-\eqref{cond-infinity} in the case $N=2$ gives rise to a corresponding solution $\tilde v$ to \eqref{mainProblem}-\eqref{cond-infinity} in general dimension $N \ge 3$ by simply setting $\tilde v(x) = v(x_{1},x_{N})$. 
\smallbreak
\item[(e)] It remains as an interesting open question whether the function $x \mapsto w_0(x_N)$ is the unique bounded positive solution to (\ref{mainProblem})-(\ref{cond-infinity}) in the case $c=\woo$. At first glance, it seems natural to establish such a uniqueness result also with the help of a sliding argument as mentioned in (b) above, but additional difficulties appear in the case $c= \woo$, and non-uniqueness remains a possibility for now. 
\smallbreak 
\end{itemize}
\end{remark}

%\medbreak
%We first point out that the uniqueness and nonexistence parts (ii) and (iii) of Theorem~\ref{th1} are proved with a variant of the sliding method based on a comparison with $x_N$-translates of the function $x \mapsto w_0(x_N)$. We shall comment on this in more detail further below. 
%As we have indicated above already, Theorem \ref{th1} (i) highlights the rich structure of solutions to (\ref{mainProblem})-(\ref{cond-infinity}) in the case $N \ge 2$, $0< c< \woo$, which is in striking contrast to the case $N =1$. We also note that it suffices to prove Theorem \ref{th1} (i) in the case $N=2$ since every positive solution $v$ to \eqref{mainProblem}-\eqref{cond-infinity} in the case $N=2$ gives rise to a corresponding solution $\tilde v$ to \eqref{mainProblem}-\eqref{cond-infinity} in general dimension $N \ge 3$ by simply setting $\tilde v(x) = v(x_{1},x_{N})$. 
%We also point out that problem (\ref{mainProblem})-(\ref{cond-infinity}) is invariant under translations and rotations parallel to the boundary $\partial \R^N_+ = \R^{N-1}$. In particular, if $N \ge 2$, any positive solution $v$ to (\ref{mainProblem})-(\ref{cond-infinity}) which is not merely a function of the $x_N$-variable already gives rise to infinitely many solutions of the type $x \mapsto v(x' + \tau,x_N)$, $\tau \in \R^{N-1}$, where, here and in the following, we write $x = (x',x_N)$ for $x \in \R^N_{+}$ with $x' \in \R^{N-1}$.

The following result provides some information on the shape of the solutions we construct.

\begin{thm} \label{th1-existence-refined}
Let $ N \geq 2$, $1 < p < 2^*-1$ and $0 < c < \woo$. Then there exists a positive solution to (\ref{mainProblem})-(\ref{cond-infinity}) of the form
  \begin{equation}
    \label{eq:solution-ansatz}
  x \mapsto w_0(x_N + t_{c,p}) + u(x),
  \end{equation}
  with a nonnegative function $u \in H^1_0(\R^N_+) \setminus \{0\} $.
\end{thm}

\medbreak
By the remarks above and since all exponents $p<\infty$ are subcritical in the case $N=2$, Theorem \ref{th1-existence-refined} implies Theorem \ref{th1} (i) in the case $N=2$ and therefore for all $N \ge 2$. It also allows us to distinguish different solution orbits under translations and rotations in $\R^{N-1}$.

\begin{cor} \label{th1-existence-refined-multiplicity}
Let $N \geq 3$, $1 < p < \frac{M+2}{M-2}$ for some $M \in \{3,\dots,N\}$ and $0 < c < \woo$. Then problem (\ref{mainProblem})-(\ref{cond-infinity}) admits at least $M+1$ geometrically distinct positive solutions.
\end{cor}

This result is a rather immediate corollary of Theorem~\ref{th1-existence-refined}. Indeed, under the given assumptions, for every dimension ${\widetilde N} \in \{2,3,\dots,M\}$, Theorem \ref{th1-existence-refined}  yields the existence of a solution to (\ref{mainProblem})-(\ref{cond-infinity}) of the form
$$ 
x \mapsto w_0(x_{\widetilde{N}} + t_{c,p}) + u(x_1,\dots,x_{{\widetilde N}-1},x_{\widetilde{N}})
$$
with a nonnegative $u \in H^1_0(\R^{{\widetilde N}}_+) \setminus \{0\}$. Clearly, these $M-1$ solutions are geometrically distinct, and they are also geometrically distinct from the two one-dimensional profile solutions.

\medbreak
It is natural to guess that the change of the solution set when passing from $c>\woo$ to $c<\woo$ is a bifurcation phenomenon. More precisely, one may guess that the solutions constructed in Theorem~\ref{th1-existence-refined} have the property that $u = u_c \to 0 \in H^1_0(\R^N_+)$ as $c \nearrow \woo$ for the functions $u$ in the ansatz (\ref{eq:solution-ansatz}). This remains an open question, and the answer could even depend on the value of $p$. We note that standard results from bifurcation theory do not apply here since the linearized problem 
$$
\left\{
\begin{aligned}
 -\Delta v + v - p |w_0(x_N)|^{p-1} v & = 0 \quad & \textup{ in } \RNp, \\  v& = 0 & \quad \textup{ on } \partial \RNp  ,
 \end{aligned}
 \right.
$$
at the parameter value $c = c_p$ has purely essential spectrum due to its invariance with respect to translations in directions parallel to the boundary $\partial \R^N_+ = \R^{N-1}$. Bifurcation from the essential spectrum has been observed succesfully in other contexts (see e.g. the survey paper \cite{stuart:1997} and the references therein), but  
there is still no general functional analytic framework which provides sufficient abstract conditions. 
 
\medbreak 
We now give some ideas of the proof of Theorem \ref{th1-existence-refined}. For this we fix $c \in (0,\woo)$ and define the functions
\[
t \mapsto z_c(t):= w_0\left(t + t_{c,p}\right)
\qquad \textup{ and } \qquad t \mapsto \widetilde{z}_c(t):= w_0\left(t -t_{c,p}\right),
\]
where $t_{c,p}$ is given in (\ref{eq:def-t-c-p}). We recall that $z_c$ and $\widetilde{z}_c$ are the unique positive solutions to \eqref{1dim} such that $z_c(0) = \widetilde{z}_c(0) = c$. Moreover, we define $u_c: \overline{\RNp} \to \R$  and $\widetilde{u}_c : \overline{\RNp} \to \R$ as
\begin{equation} \label{u0}
u_c(x) := z_c(x_N) \quad \textup{ and } \quad \widetilde{u}_c(x):= \widetilde{z}_c(x_N),
\end{equation}
and we directly notice that $u_c$ and $\widetilde{u}_c$ are both solutions to \eqref{mainProblem}-\eqref{cond-infinity}. Furthermore, it follows that
\begin{equation} \label{u0-u0tilda-infinity}
u_c(x) = O(e^{-x_N}) \quad \textup{ and } \quad \widetilde{u}_c(x) = O(e^{-x_N}), \quad \textup{ as } x_N \to \infty.
\end{equation}
Proving Theorem~\ref{th1-existence-refined} now amounts to find a nonnegative solution $u \in H_0^1(\RNp) \setminus \{0\}$ to the non-autonomous Schr\"odinger type equation
\begin{equation} \label{non-autonomous}
-\Delta u + u = f(x,u), \quad u \in H_0^1(\RNp),
\end{equation}
with
\begin{equation} \label{f}
f(x,s):= |u_c(x)+s|^{\,p-1}(u_c(x)+s) - (u_c(x))^{\,p},
\end{equation}
because in this case $v = u_c + u$ is of the form (\ref{eq:solution-ansatz}), solves (\ref{mainProblem}) and it is easy to see that also satisfies \eqref{cond-infinity}. Since we are interested in finding non-negative solutions to \eqref{non-autonomous}, we truncate the nonlinearity and define
\begin{equation}
  \label{def-g}
g(x,s):= (u_c(x)+s^+)^{\,p} - (u_c(x))^{\,p} = \left\{
\begin{aligned}
& f(x,s), \quad & \textup{ if } s \geq 0,\\
& 0, & \textup{ if } s \leq 0,
\end{aligned}
\right.
\end{equation}
with $f$ given in \eqref{f}. We then consider the auxiliary problem
\begin{equation} \label{non-autonomous-positive}
-\Delta u + u = g(x,u), \quad u \in H_0^1(\RNp).
\end{equation}
Considering $u^{-} \in H_0^1(\RNp)$ as test function in \eqref{non-autonomous-positive}, one can easily check that every solution to \eqref{non-autonomous-positive} is non-negative and so, that every solution to \eqref{non-autonomous-positive} is a non-negative solution to \eqref{non-autonomous}.
It might be worth pointing out that the one-dimensional function $\widetilde{u}:= \widetilde{u}_c - u_c \in C^2(\RNp)$ is a positive solution to
the equation in (\ref{non-autonomous-positive}) and also satisfies $\widetilde{u}  = 0$ on $\partial \RNp$. However, $\widetilde{u} \not  \in H_0^1(\RNp)$ since it only depends on the $x_N$ variable. Hence, $\widetilde{u}$ is not a solution to \eqref{non-autonomous-positive}. 

We shall look for a non-trivial solution to \eqref{non-autonomous-positive} as a critical point of the associated functional
\begin{equation} \label{E}
E: H_0^1(\RNp) \to \R,\qquad  E(u) = \frac{1}{2} \int_{\RNp} \left( |\nabla u|^2 + u^2\right) dx - \int_{\RNp} G(x,u) dx,
\end{equation}
where
\begin{equation}
  \label{eq:def-G}
G(x,u):= \int_0^u g(x,s) ds = \frac{1}{p+1}\bigl( (u_c+u^{+})^{p+1}-u_c^{p+1} - (p+1) u_c^p u^{+} \bigr).  
\end{equation}
More precisely, we are going to prove the existence of a non-trivial critical point of mountain pass type. This requires new and %somewhat
subtle estimates. 
%First, we mention that for proving that zero is a strict local minimum of $E$ we need to argue by contradiction and rely eventually on the strict monotonicity of the function $u_c$ in the $x_N$-direction. 
%More importantly,
The key difficulties in the variational approach are the non-standard shape of the nonlinearity $g$ in (\ref{non-autonomous-positive}) and the lack of compactness due to the unboundedness of $\RNp$. To overcome these difficulties, we need new estimates
within the analysis of Cerami sequences and for comparing the mountain pass energy value for $E$ with the corresponding one of the limit energy functional
\begin{equation} \label{Einfty-intro}
E_\infty: H^1(\R^N) \to \R, \qquad E_{\infty}(u) = \frac{1}{2}\int_{\RN} \big( |\gradu|^2 + u^2\big)\,dx - \frac{1}{p+1} \int_{\RN} |u|^{\,p+1} dx.
\end{equation}
In particular, we shall use the asymptotic decay properties of the unique positive radial solution to (\ref{problem-RN-intro}) in order to build suitable test functions.

We wish to mention two further open problems at this stage. First, one may ask whether Theorem~\ref{th1-existence-refined} extends to the critical case $N \ge 3$, $p=2^*-1$. In this case, the mountain pass geometry of the functional $E$ remains, but the lack of compactness is more severe as it is not only caused by the unboundedness of $\R^N_+$ but also by possible point concentration of bounded Cerami sequences. Indeed, while the limit energy functional $E_\infty$ in (\ref{Einfty-intro}) does not admit critical points in the case $p=2^*$ by Pohozaev's identity (see e.g. \cite[Corollary B.4]{Wi1996}), rescaling bounded Cerami sequences around possible concentration points leads to critical points of the Yamabe functional  $u \mapsto \frac{1}{2}\int_{\RN}  |\gradu|^2 \,dx - \frac{1}{2^*} \int_{\RN} |u|^{2^*} dx$. It therefore seems natural to build test functions from translated and concentrated instantons $x \mapsto [N(N-2)\eps^2]^{\frac{N-2}{4}} \bigl(\eps^2+|x-x_*|^2\bigr)^{-\frac{N-2}{2}}$ in order to estimate the mountain pass energy. However, since our estimates rely on the precise exponential decay rate of the unique positive radial solution to (\ref{problem-RN-intro}) as given in (\ref{behaviour-gs-infinity}), they do not apply to instantons. Hence the case $p=2^*$ remains a problem for future research.

The second open problem concerns the existence of solutions of the form $x \mapsto w_0(x_N - t_{c,p}) + u(x)$ with $u \in H^1_0(\R^N_+) \setminus \{0\}$. In this case, $u_c$ has to be replaced by $\widetilde{u}_c$ in the definition of the nonlinearity $G$ in (\ref{eq:def-G}). One may observe that the mountain pass geometry is lost in this case, and the presence of essential spectrum leads to additional difficulties which seem hard to deal with.

\medbreak  
 
We now comment on the proofs of the non-existence part (ii) of Theorem \ref{th1}. We argue by contradiction and  use a suitable modification of the so-called sliding method introduced by H. Berestycki and L. Nirenberg and further developed by H. Berestycki, L. Caffarelli and L. Nirenberg among others. Specifically, our proofs  are inspired by \cite[Section 2]{Fa2003} and \cite[Section 4]{BeHaMo2000}.

\medbreak

We finally comment on the boundedness of positive solutions to \eqref{mainProblem}-\eqref{cond-infinity}. As stated in the following proposition, all the positive solutions to \eqref{mainProblem}-\eqref{cond-infinity} are bounded in the case where $1 < p < 2^{*}-1$. Hence, the fact that we are considering bounded solutions is not a restriction in this case.

\begin{prop} \label{boundedness-of-the-solutions}
Let $N \geq 2$, $1 < p < 2^*-1$ and $c > 0$. Any positive $v \in C^2(\RNp) \cap C(\overline{\RNp})$-solution to \eqref{mainProblem}-\eqref{cond-infinity} belongs to $L^{\infty}(\RNp)$.
\end{prop}

The proof of Proposition~\ref{boundedness-of-the-solutions} follows by a rather standard blow up argument based on the doubling lemma by P. Pol\'{a}\v{c}ik, P. Quittner and P. Souplet in \cite{PoQuiSou2007}. For the convenience of the reader, we include the proof in Section~\ref{section-existence} below.

\subsection*{Organization of the paper}
%The paper is organized as follows. 
In Section~\ref{section-preliminaries}, we collect estimates related to the nonlinearity $g$ in (\ref{def-g}) and the functional $E$ associated with (\ref{non-autonomous-positive}). With the help of these estimates, we establish the mountain pass geometry of $E$ in Section~\ref{sec:mount-pass-geom}, and we show that Cerami sequences at nontrivial energy levels are bounded and admit nontrivial weak limits after suitable translation. In Section~\ref{section-energy-estimate}, we then prove a key energy estimate which shows that, in dimensions $N \ge 2$, the mountain pass energy of the functional $E$ is strictly smaller than the corresponding one for the limit energy functional $E_\infty$ given in (\ref{Einfty-intro}). With the help of this energy estimate, we then complete the proof of Theorem~\ref{th1-existence-refined} in Section~\ref{section-existence}. Finally, we give the proof of Theorem~\ref{th1} (ii) in Section~\ref{section-non-existence-uniqueness}.

\subsection*{Notation} For  $1\leq p<\infty,$ we let $\|\cdot \|_{L^p(\RNp)}$ denote the standard norm on the usual Lebesgue space $L^p(\RNp)$. The Sobolev space $H_0^1(\RNp)$ is endowed with the standard norm
\[ \|u\|^2 = \int_{\RNp} \big(\,|\gradu|^2 + |u|^2 \big)\, dx.\]
Also, for a function $v$, we define $v^{+} := \max\{v,0\}$ and $v^{-}:= \max\{-v,0\}$ and we write $x = (x',x_N)$ for $x \in \RNp$ with $x' \in \R^{N-1}$. We denote by $'\rightarrow'$, respectively by $'\rightharpoonup'$, the strong convergence, respectively the  weak convergence in corresponding space and denote by $B_R(x)$ the open ball in $\R^N$ of center $x$ and radius $R>0.$ Also, we shall denote by $C_i>0$ different constants which may vary from line to line but are not essential to the analysis of the problem. Finally, at various places, we have to distinguish the cases $p \le 2$ and $p>2$. For this it is convenient to introduce the special constant
$$
1_{\{p>2\}}:= \left\{
  \begin{aligned}
    &0,&&\qquad p \le 2,\\
    &1,&&\qquad p > 2.
  \end{aligned}
\right.
$$ 

\subsection*{Acknowledgements}  The authors wish to thank the anonymous referees for their valuable comments and corrections. Part of this work was done while the first  author was visiting the Goethe-Universit\"at Frankfurt. He wishes to thank his hosts for the warm hospitality and the financial support. 

\section{Preliminaries} \label{section-preliminaries}

\noindent In this section we collect some estimates related to the transformed nonlinearity $g$ defined in (\ref{def-g}), its primitive $G$ and the functional $E$ defined in (\ref{E}). For this we fix, throughout Sections~\ref{section-preliminaries}--Section~\ref{section-existence}, $c \in (0,\woo)$, $p \in (1,2^{*}-1)$, and we let $u_c$ be given in
(\ref{u0}). We recall that we have the uniform estimate
\begin{equation}
\label{u-0-uniform-est}
0 \le u_c \le c_p \qquad \text{in  $\RNp$.}
\end{equation}
We start with an elementary inequality for nonnegative real numbers which will be used in the energy estimates in Section \ref{section-energy-estimate} below.

\begin{lem} \label{inequality-q-refined}
For every $q > 2$ there exists $\kappa_q>0$ with 
\begin{equation}
  \label{elem-lemma-claim}
  (a+b)^q - a^q-b^q \geq qa^{q-1}b +\kappa_q a b^{q-1} \qquad \text{for all $a,b \ge 0$.}
\end{equation}
\end{lem}

\begin{remark}
\label{elem-est-remark}
If $q \ge 3$, then (\ref{elem-lemma-claim}) holds in symmetric form with $\kappa_q = q$, see e.g. \cite[Theorem 1]{Ja2014}. If $q \in (2,3)$, it is easy to see that one has to choose $\kappa_q < q$.
\end{remark}

\begin{proof}[Proof of Lemma~\ref{inequality-q-refined}]
  We first note that, since $q-1>1$, we have, by convexity of the function $\tau \mapsto (1+\tau)^{q-1}$, 
  \begin{equation}
    \label{eq:first-est-elem-lemma}
    (s+t)^{q-1}-\bigl(s^{q-1}+t^{q-1}\bigr)= s^{q-1}\Bigl[\bigl(1+\frac{t}{s}\bigr)^{q-1}- \bigl(1+\bigl(\frac{t}{s}\bigr)^{q-1}\bigr)\Bigr] \ge s^{q-1}\Bigl((q-1)\frac{t}{s} - \bigl(\frac{t}{s}\bigr)^{q-1}\Bigr)
    = (q-1)ts^{q-2} - t^{q-1}
  \end{equation}
  for $s \ge t >0$. Now, to prove the claim, it suffices to consider $a,b>0$, since the inequality holds trivially if $a=0$ or $b=0$. Moreover, it suffices to prove that the inequality holds for $b \ge a>0$ with some $\kappa_q \in (0,q]$, since then it also follows for arbitrary $a,b >0$. For fixed $a>0$, we consider the function
  $$
  \ell: [0,\infty) \to \R, \qquad \ell(t)= (a+t)^q - a^q-t^q - qa^{q-1}t
  $$
  Then we have $\ell(0)=0$ and 
  $$
  \ell'(t)= q \Bigl[(a+t)^{q-1}-\bigl(a^{q-1}+t^{q-1}\bigr)\Bigr]
  $$
  Consequently, by (\ref{eq:first-est-elem-lemma}) we have, for $b \ge a$,  
   \begin{align*}
   \ell(b)&=  \int_0^a \ell'(t)\,dt+ \int_a^b \ell'(t)\,dt \ge q \Bigl[ \int_0^a \bigl( (q-1)ta^{q-2} - t^{q-1}\bigr)dt 
+ \int_a^b \bigl( (q-1)a t^{q-2} - a^{q-1}\bigr)dt\Bigr]\\
       &= qa  \bigl( \kappa_{q,1} a^{q-1} + b^{q-1} - a^{q-2}b \bigr)\qquad \text{with $\quad \kappa_{q,1} :=\frac{q-1}{2}-\frac{1}{q} >0$.}
    \end{align*}
    Since, by Young's inequality,
    $$
    a^{q-2}b \le \frac{q-2}{q-1}a^{q-1} + \frac{1}{q-1}b^{q-1},
    $$
 we deduce that 
$$
\ell(b) \ge   qa \bigl[ \bigl(\kappa_{q,1}-\frac{q-2}{q-1}\bigr) a^{q-1} + \frac{q-2}{q-1} b^{q-1}\bigr].
$$
If $\kappa_{q,1} \ge \frac{q-2}{q-1}$, we conclude that $\ell(b) \ge    \frac{q(q-2)}{q-1} a b^{q-1}$. On the other hand, if $0< \kappa_{q,1} < \frac{q-2}{q-1}$, we use again that $b \ge a$ and conclude that
$$
\ell(b) \ge   qa \bigl[ \bigl(\kappa_{q,1}-\frac{q-2}{q-1}\bigr) b^{q-1} + \frac{q-2}{q-1} b^{q-1}\bigr]= q \kappa_{q,1} a b^{q-1}.
$$
Hence, (\ref{elem-lemma-claim}) holds for $b \ge a >0$ with $\kappa_q = q \min \{ \frac{q-2}{q-1}, \kappa_{q,1}\} \in (0,q)$. The proof is finished.
  \end{proof}

Next we provide basic but important estimates for the nonlinearity $g$ defined in (\ref{def-g}) and its primitive $G$.
 
\medbreak 
\begin{lem} \label{G-g-estimate} $ $
  \begin{itemize}
  \item[(i)] 
  For $(x,s) \in \RNp \times \R$ we have
  \begin{equation}
    \label{eq:g-1-estimate}
 0 \le g(x,s)- s^+ p u_c(x)^{p-1} \le C_{1,p}\,  [s^+]^{\,p} + 1_{\{p>2\}}C_{2,p}\,  [s^+]^2,
  \end{equation}
  and
  \begin{equation}
    \label{eq:G-1-estimate}
   0 \le G(x,s)- \frac{p}{2}[s^+]^2 u_c(x)^{p-1} \le \frac{C_{1,p}}{p+1}[s^+]^{p+1} + 1_{\{p>2\}}\frac{C_{2,p}}{3} [s^+]^3,
  \end{equation}
  with $C_{1,p}:= 1+2^{p-3} p(p-1)$ and $C_{2,p}:=p(p-1)2^{p-3}\woo^{p-2}$.\\
\item[(ii)] Let
  $$
  H(x,s) := \frac{1}{2}g(x,s)s -G(x,s)\qquad \text{for $x \in \RNp$, $s \in \R$.}
  $$
  Then we have
  \begin{equation}
    \label{eq:H-lower-est}
H(x,s) \ge \max \Bigl\{ \:0\:,\: \frac{p-1}{2(p+1)}\, [s^+]^{p+1} - u_c(x)^{p-1} D_{1,p}[s^+]^2 - u_c(x)\, 1_{\{p>2\}} D_{2,p} [s^{+}]^p \Bigr\}  
  \end{equation}
  with $D_{1,p}:= \frac{p(p-1)}{p+1}\bigl(1+2^{p-2}\bigr)$ and $D_{2,p}:= \frac{p2^{p-2}}{p+1}$.\\[0.2cm]
Moreover, the function $H(x,\cdot)$ is non-decreasing in $s$ for every $x \in \RNp$.  
\end{itemize}
\end{lem}

\begin{remark}
The constants $C_{i,p}$ and $D_{i,p}$, $i = 1,2$, are not optimal. However, this choice simplifies the presentation. Moreover, they do not play an important role in our proofs below.
\end{remark}

\begin{proof}[Proof of Lemma \ref{G-g-estimate}]
(i) Since $g(\cdot,s) \equiv 0$ and $G(\cdot,s) \equiv 0$ for $s \le 0$, it suffices to consider $s > 0$. Fix $x \in \RNp$. Since $g(x,\cdot)$ is of class $C^1$ on $(0,\infty)$, we have
  \begin{equation}
    \label{eq:g-1-estimate-step-1}
g(x,s)=(u_c(x)+s)^{p}-u_c(x)^p= s p u_c(x)^{p-1} + p\int_0^s \bigl[(u_c(x)+\tau)^{p-1}-u_c(x)^{p-1}\bigr]\,d\tau, \qquad \text{for $s > 0$.}
\end{equation}
We now distinguish two cases. If $p \in (1,2]$, we have
$$
0 \le (a +\tau)^{p-1} - a^{p-1} \le \tau^{p-1} \qquad \text{for $\tau>0,\: a \ge 0$},
$$
and therefore, if $p \in (1,2]$,
\begin{equation}
  \label{eq:g-est-p-less-2}
0 \le   g(x,s) -  s p u_c(x)^{p-1} \le p \int_0^s \tau^{p-1}\,d\tau \le s^p, \qquad \text{for $s \ge 0$.}
\end{equation}
If $p>2$, we have 
$$
0 \le (a +\tau)^{p-1} - a^{p-1} \le (p-1)\tau (a+\tau)^{p-2},  \qquad \text{for $\tau>0,\: a \ge 0$},
$$
by the convexity of the function $a \mapsto (a +\tau)^{p-1}$ and therefore, using also (\ref{u-0-uniform-est}), 
  \begin{align*}
    0 &\le  g(x,s) -  s p u_c(x)^{p-1} \le p\int_0^s \bigl[(u_c(x)+\tau)^{p-1}-u_c(x)^{p-1}\bigr]\,d\tau \\
      &\le p (p-1)\int_0^s \tau (u_c(x)+\tau)^{p-2}\,d\tau \le \frac{p(p-1)}{2}s^2 (u_c(x)+s)^{p-2} \le \frac{p(p-1)}{2}s^2 (\woo +s)^{p-2}.
  \end{align*}
Note also that, since $p > 2$,
  \begin{equation}
    \label{eq:p-2-est}
 (\woo +s)^{p-2} \le \bigl(2 \max\{\woo,s\}\bigr)^{p-2} \le (2\woo)^{p-2} + (2s)^{p-2}, \qquad \text{for $s \ge 0$.}
  \end{equation}
Consequently, if $p > 2$,
\begin{equation}
\label{eq:g-est-p-greater-2}
    0 \le  g(x,s) -  s p u_c(x)^{p-1}  \le 2^{p-3} p(p-1)s^p + p(p-1)2^{p-3}\woo^{p-2} s^2, \qquad \text{for $s \ge 0$.}
\end{equation}
Now (\ref{eq:g-1-estimate}) follows by combining (\ref{eq:g-est-p-less-2}) and (\ref{eq:g-est-p-greater-2}). Moreover, (\ref{eq:G-1-estimate}) follows by integrating (\ref{eq:g-1-estimate}).
 
\medbreak
(ii) We first note that $H(x,s) \equiv 0$ for all $s \leq 0$. Thus, we just have to prove the result for $s > 0$. Directly observe that, for all $x \in \RNp$, we have $H(x,\cdot) \in C^1(\R)$ and
$$
\begin{aligned}
\frac{\partial}{\partial s} H(x,s) & = \frac{p}{2}(u_c(x)+s)^{p-1}s + \frac{1}{2} (u_c(x)+s)^p - \frac{1}{2} u_c^p(x)- (u_c(x)+s)^p + u_c^p(x) \\
& = \frac{1}{2} \left[ p (u_c(x)+s)^{p-1} s - (u_c(x)+s)^p+ u_c^p(x) \right] \qquad \text{for $(x,s) \in \RNp \times (0,+\infty).$}
\end{aligned}
$$
On the other hand, since $p>1$, we have, by the mean value theorem, 
$$
(u_c(x)+s)^p - u_c^p(x) \leq p (u_c(x)+s)^{p-1} s  \qquad \text{for $(x,s) \in \RNp \times (0,+\infty).$}
$$ 
Hence $ \frac{\partial}{\partial s} H(x,s) \geq 0$ for all $s > 0$, so the function $H(x,\cdot)$ is non-decrasing in $s \in [0,\infty)$. This also implies that $H(x,s) \ge 0$ for $s \ge 0$. It thus remains to prove (\ref{eq:H-lower-est}) for $s \ge 0$. For this we first note that
\begin{align*}
H(x,s) =\frac{g(x,s)s}{2}-G(x,s) &=  \frac{1}{2}\Bigl((u_c(x)+s)^p - u_c(x)^p\Bigr)s - \frac{1}{p+1}\Bigl((u_c(x)+s)^{p+1} - u_c(x)^{p+1} - (p+1) u_c(x)^{p}s \Bigr)\\
&=  \Bigl(\frac{1}{2}-\frac{1}{p+1}\Bigr)\Bigl((u_c(x)+s)^p - u_c(x)^p\Bigr)s - \frac{1}{p+1}\Bigl((u_c(x)+s)^{p} - u_c(x)^{p} - p u_c(x)^{p-1}s \Bigr)u_c(x)\\  
&\ge   \frac{p-1}{2(p+1)}\,s^{p+1} - \frac{1}{p+1}u_c(x) \bigl(g(x,s)- s p u_c(x)^{p-1}\bigr) . 
\end{align*}
It therefore remains to show that 
\begin{equation}
  \label{g-special-est}
u_c(x) \bigl(g(x,s)- s p u_c^{p-1}(x)\bigr) \le (p+1) \bigl( u_c^{p-1}(x)\, D_{1,p} s^2 + u_c(x)\, 1_{\{p > 2\}} D_{2,p} s^p \bigr), \quad \textup{ for } s > 0. 
\end{equation}
By (\ref{eq:g-1-estimate-step-1}) and integration by parts we have
\begin{align}
u_c(x) \bigl(g(x,s)- s p u_c^{p-1}(x)\bigr)&= p u_c(x)\int_0^s \bigl[(u_c(x)+\tau)^{p-1}-u_c(x)^{p-1}\bigr]\,d\tau \nonumber\\ 
&= p(p-1) u_c(x) \int_0^s (s-\tau) \bigl[(u_c(x)+\tau)^{p-2}\bigr]\,d\tau. \label{intermediate-special-est}  
\end{align}
If $p \in (1,2]$, we have $(u_c(x)+\tau)^{p-2} \le u_c(x)^{p-2}$ for $\tau>0$ and therefore
\begin{equation}
  \label{g-special-est-1}
u_c(x) \bigl(g(x,s)- s p u_c(x)^{p-1}\bigr) \le  p(p-1) u_c^{p-1}(x) \int_0^s (s-\tau) \,d\tau \le p(p-1) u_c^{p-1}(x)s^2.
\end{equation}
If $p >2$, arguing as (\ref{eq:p-2-est}), we have
$$
(u_c(x)+\tau)^{p-2}  \le  2^{p-2} \bigr( u_c^{p-2}(x) + \tau^{p-2} \bigr),\qquad \text{for $\tau \ge 0$},
$$
and therefore (\ref{intermediate-special-est}) yields
\begin{equation} \label{special-est-2}
\begin{aligned}
u_c(x) \bigl(g(x,s)- s p u_c(x)^{p-1}\bigr) &\le  2^{p-2} p(p-1) u_c(x) \int_0^s (s-\tau)\bigl[u_c^{p-2}(x) + \tau^{p-2}\bigr] \,d\tau  \\
&= 2^{p-2} p(p-1) u_c(x) \left( u_c^{p-2}(x) \int_0^s (s-\tau) d\tau + \int_0^s (s-\tau) \tau^{p-2} d\tau\right)\\
&\le  
2^{p-2} p (p-1) u_c^{p-1}(x)\, s^2 + 2^{p-2} p u_c(x)\, s^p, \quad \textup{ for } s > 0.
\end{aligned}
\end{equation}
Now (\ref{g-special-est}) follows by combining (\ref{g-special-est-1}) and (\ref{special-est-2}). The proof is finished. 
\end{proof}

\begin{remark}
\label{functional-diff}  $ $
\begin{itemize}
\item[(a)] From the growth estimates given in Lemma~\ref{G-g-estimate} (i) and the fact that $g$ is continuous, it follows in a standard way that the functional $E$ is well-defined on $H_0^1(\RNp)$ and of class $C^1$.
\item[(b)]
Part (ii) of Lemma~\ref{G-g-estimate} will be useful in the analysis of Cerami sequences of the functional $E$, see Section~\ref{sec:mount-pass-geom} below.
\end{itemize}
\end{remark}

Next, we consider the quadratic form $q_c: H_0^1(\RNp)  \to \R$ given by 
\begin{equation} \label{qepsilon}
q_{c}(u) := \int_{\RNp} \big(|\gradu|^2 + V_c(x) u^2\big)\,dx,
\end{equation}
with
\begin{equation} \label{Vepsilon}
\qquad V_c(x) := 1 - p u_c^{p-1}(x_N)\, \in\, L^{\infty}(\RNp).
\end{equation}
As we show in the following lemma, $q_{c}$ is positive definite on $H^1_0(\R^N_+)$.

\begin{prop} \label{positive-definite-main-proposition}
We have 
  \begin{equation}
\label{main-proposition-equation}
\widetilde q_c := \inf_{u \in H^1_0(\R^N_+) \setminus \{0\}} \frac{q_{c}(u)}{\|u\|^2}>0.
\end{equation}
\end{prop}

\bigbreak
\begin{remark} $ $
\begin{itemize}
\item[(a)] Recall that we are using the shortened notation $\|u\|^2  = \int_{\RNp} \big(|\nabla u|^2 + |u|^2\big)\, dx$.
\item[(b)] From Proposition \ref{positive-definite-main-proposition} it follows that $(q_c(\cdot))^{1/2}$ is an equivalent norm to $\|\cdot\|$ in $H_0^1(\RNp)$.
\end{itemize}
\end{remark}

\begin{proof}[Proof of Proposition \ref{positive-definite-main-proposition}]
Since $V_{c} \in L^\infty(\R^N_+)$, it suffices to show there exists $C>0$ such that
\begin{equation}
\label{main-proposition-equation-1}
q_{c}(u) \ge C \|u\|_{L^2(\R^N_+)}^2 \qquad \text{for all $u \in H^1_0(\R^N_+).$}
\end{equation}   
Indeed, if (\ref{main-proposition-equation-1}) holds, then for $\delta \in (0,1)$ we have  
\begin{align*}
q_{c}(u) & \ge \delta q_{c}(u)+ (1-\delta)C \|u\|^2_{L^2(\R^N_+)} \ge 
\delta \|u\|^2 + \bigl[(1-\delta)C - \delta\bigl(1 +\|V_{c}\|_{L^\infty(\R^N_+)}\bigr)\bigr]\,\|u\|^2_{L^2(\R^N_+)}.
\end{align*}
Choosing $\delta$ sufficiently small, we have $(1-\delta)C - \delta\bigl(1 +\|V_{c}\|_{L^\infty(\R^N_+)}\bigr)\ge 0$ and therefore (\ref{main-proposition-equation}) holds with $\widetilde q_c \ge \delta$.\\
To show (\ref{main-proposition-equation-1}), we first consider the case $N= 1$. Arguing by contradiction, we assume that 
$$
\lambda:= \inf \big\{ q_{c}(u): u \in H_0^1(\R_{+}) \textup{ and } \|u\|_{L^2(\R_{+})} = 1 \big\} \leq 0,
$$
(note that $\lambda>-\infty$ since $V_{c} \in L^\infty(\R_+)$). 
Then, there exists a sequence $(u_n)_n$ such that $\|u_n\|_{L^2(\R_+)}= 1$ for all $n \in \N$ and $q_{c}(u_n) \to \lambda$ as $n \to \infty$. Hence, $(u_n)_n$ is a bounded sequence in $H_0^1(\R_+)$, and thus $u_n \rightharpoonup u_*$ weakly in $H_0^1(\R_+)$ after passing to a subsequence. Moreover, with $v_n:= u_n- u_*$, we have
$v_n \weakto 0$ in $H_0^1(\R_+)$ and therefore $v_n \to 0$ in $L^2_{loc}(\R_{+})$. 
Since $V_{c}(t) \to 1$ as $t \to \infty$, this implies that 
$$
q_{c}(v_n) \ge \int_{\R_+}V_{c}(t) v_n^2\,dt \ge \|v_n\|_{L^2(\R_{+})}^2 + o(1), \qquad \text{as $n \to \infty$},
$$
and therefore 
\begin{align*}
\lambda + o(1)&= q_{c}(u_n)= q_{c}(u_*) + q_{c}(v_n)+ o(1) \\
& \ge \lambda \|u_*\|_{L^2(\R_{+})}^2 + \|v_n\|_{L^2(\R_{+})}^2+o(1) = \lambda\big(\|u_*\|_{L^2(\R_{+})}^2+ \|v_n\|_{L^2(\R_{+})}^2\big) + (1-\lambda)\|v_n\|_{L^2(\R_{+})}^2 + o(1)\\
&= \lambda + (1-\lambda)\|v_n\|_{L^2(\R_{+})}^2 +o(1)
\end{align*}
It thus follows that $v_n \to 0$ in $L^2(\R_{+})$ and hence $u_n \to u_*$ in $L^2(\R_+)$, which yields that $\|u_*\|_{L^2}=1$. Moreover, by weak lower semicontinuity of $q_{c}$ and the definition of $\lambda$, it follows that $q_{c}(u_*)=\lambda$, so $u_*$ is a constrained minimizer for $q_{c}$. A standard argument (based on replacing $u_*$ by $|u_*|$) shows that $u_* \in H_{0}^1(\R_+)$ is a positive or negative solution of 
$$
-u_*'' + V_{c}(t)u_* = \lambda u_* \quad \text{in $\R_+$}, \qquad u_*(0)=0. 
$$
Without loss of generality, we may assume that $u_*$ is positive, which implies that $u_*'(0)>0$. We also recall that $w_*:= -u_{c}'$ satisfies 
$$
-w_*'' + V_{c}(t)w_* = 0 \quad \text{in $\R_+$}, \qquad w_*>0 \quad \text{in $\overline{\R_+}$}. 
$$
Consequently, we have 
$$
0 \leq - \lambda \int_{\R^+}w_* u_* \,dx = \int_{\R^+}  \bigl(w_* u_*''  - u_* w_*''\bigr)\,dx =- w_*(0) u_*'(0)<0,
$$
a contradiction. Hence, we conclude that \eqref{main-proposition-equation-1} holds in the case $N=1$. To show (\ref{main-proposition-equation-1}) for general $N\geq2$, we remark that, by density, we only have to show it for 
$u \in C^\infty_c(\R^N_+)$. For any such function we then have, writing 
$x = (x',t) \in \R^N_+$ with $x' \in \R^{N-1}$, $t>0$:
\begin{align*}
q_{c}(u) &\ge \int_{\R^N_+}\Bigl(|\partial_t u|^2 + V_{c} u^2\Bigr)\,dx \\
&= \int_{\R^{N-1}}  \int_{\R_+}\Bigl(|\partial_t u(x',t)|^2 + V_c u^2(x',t)\Bigr)\,dt dx' \\
& \ge C  \int_{\R^{N-1}} \int_{\R_+}u^2(x',t)\,dt dx' = C \|u\|^2_{L^2(\R^N_+)}.
\end{align*}
Here we have used the result in the case $N=1$ and the fact that $u(x', \cdot) \in C^\infty_c(\R_+) \subset H^1_0(\R_+)$ for every $x' \in \R^{N-1}$. 
We thus have proved (\ref{main-proposition-equation-1}) for general $N \ge 1$, and the proof is complete.
\end{proof}

Having at hand Proposition \ref{positive-definite-main-proposition}, we prove a lower estimate on the functional $E$ given in \eqref{E} that will be useful at several points below. 

\begin{cor} \label{lemma-taylor}
We have
$$
E(u) \geq \frac{\widetilde q_c}{2} \|u\|^2 - \frac{C_{1,p}}{p+1} \|u\|_{L^{p+1}(\RNp)}^{p+1} - 1_{\{p>2\}}  \frac{C_{2,p}}{3} \|u\|_{L^{3}(\RNp)}^{3}\qquad \text{for all $u \in H_0^1(\RNp).$}
$$
with $\widetilde q_c$ given in \eqref{main-proposition-equation} and $C_{1,p},C_{2,p}$ given in Lemma~\ref{G-g-estimate} (i).
\end{cor}

\begin{proof}
 For all $u \in H_0^1(\RNp)$, we have, by (\ref{eq:G-1-estimate}) and Proposition~\ref{positive-definite-main-proposition},
\begin{align*}
E(u) & = \frac{1}{2} \|u\|^2 - \int_{\RNp} G(x,u) dx  \\
     & \ge \frac{1}{2}\|u\|^2 - \frac{1}{2} \int_{\RNp} p u_c^{p-1} (u^+)^2 dx -\frac{C_{1,p}}{p+1} \|u^+\|_{L^{p+1}(\RNp)}^{p+1} - 1_{\{p>2\}}\frac{C_{2,p}}{3} \|u^+\|_{L^{3}(\RNp)}^{3}\\
     & \ge \frac{1}{2}q_c(u)  -\frac{C_{1,p}}{p+1} \|u\|_{L^{p+1}(\RNp)}^{p+1} - 1_{\{p>2\}}\frac{C_{2,p}}{3} \|u\|_{L^{3}(\RNp)}^{3}\\
  &\ge   \frac{\widetilde q_c}{2}\|u\|^2 - \frac{C_{1,p}}{p+1}\|u\|_{L^{p+1}(\RNp)}^{p+1} - 1_{\{p>2\}}\frac{C_{2,p}}{3} \|u\|_{L^{3}(\RNp)}^{3}.
\end{align*}
\end{proof}

\section{Mountain-pass geometry and boundedness of the Cerami sequences}
\label{sec:mount-pass-geom}

\noindent This section is devoted to show that the functional $E$ has a Mountain-pass geometry and that, for any $d \in \R$, the Cerami sequences for $E$ and level $d$ are bounded. We keep using the notation of the introduction and of Section~\ref{section-preliminaries}, which depends on the fixed quantities $c \in (0,\woo)$ and $p \in (1,2^{*}-1)$. We begin by proving that the functional $E$ has indeed a Mountain-pass geometry.

\begin{lem}\label{mpGeometry} The functional $E$ has the following properties.
\begin{itemize}
\item[(i)] $E(0) = 0$.
\item[(ii)] There exist $\rho_0 > 0$ and $\delta_0 > 0$ such that $E(u) \geq \delta_{0}$ for all $u \in H_0^1(\RNp)$ such that $\|u\| = \rho_0$.
\item[(iii)] There exists $\psi \in H_0^1(\RNp)$ such that $\|\psi\| > \rho_0$ and $E(\psi) < 0$.
\end{itemize}
\end{lem}
\begin{proof}
Since (i) is obvious, we concentrate on proving (ii) and (iii). We first prove (ii). Let $u \in H_0^1(\RNp)$ with $\|u\| = \rho_0$. By Corollary \ref{lemma-taylor}, we have 
$$
E(u) \geq \frac{\widetilde q_c}{2}\rho_0^2 - \frac{C_{1,p}}{p+1} \|u\|_{L^{p+1}(\RNp)}^{p+1} - 1_{\{p>2\}}\frac{C_{2,p}}{3} \|u\|_{L^3(\RNp)}^3.
$$
Applying then Sobolev embeddings, we deduce that
$$
E(u) \geq \frac{\widetilde q_c}{2}\rho_0^2 - C  \big(\rho_0^{p+1} + 1_{\{p>2\}} \rho_0^3 \big),
$$
with a constant $C>0$. Since $p>1$, Claim (ii) follows by taking $\rho_0$ sufficiently small. It then remains to prove (iii). Let $\varphi \in C_c^{\infty}(\RNp)$ with $\varphi \gneqq 0$ and $\psi := t \varphi$ with $t \in (0,+\infty)$. Directly observe that
\begin{equation*}
E(\psi)  = \frac{t^2}{2} \|\varphi\|^2 + t \int_{\RNp} u_c^p \varphi dx - \frac{1}{p+1} \int_{\RNp} \left((u_c+t\varphi)^{p+1} - u_c^{p+1}\right)dx
\end{equation*}
Then, since
$$ \int_{\RNp} \left((u_c+t\varphi)^{p+1}-u_c^{p+1} \right) dx \geq t^{p+1} \|\varphi\|_{L^{p+1}(\RNp)}^{p+1},$$
we have that
$$
E(\psi) \leq \frac{t^2}{2}\|\varphi\|^2+tc^p\|\varphi\|_{L^1(\RNp)} - \frac{t^{p+1}}{p+1} \|\varphi\|_{L^{p+1}(\RNp)}^{p+1}.
$$
Claim (iii) follows taking $t$ sufficiently large and thus the proof is complete.
\end{proof}

We now prove the boundedness of Cerami sequences of the functional $E$. 

\begin{prop} \label{boundednessPSCerami}
Cerami sequences for $E$ at any level $d \in \R$ are bounded.
\end{prop}
 
\begin{remark} $ $
\begin{itemize}
\item[(a)] Recall that $(\varphi_n)_n \subset H_0^1(\RNp)$ is a Cerami sequence for $E$ at level $d \in \R$ if
\[ E(\varphi_n) \to d \quad \textup{ and } \quad (1+\|\varphi_n\|)\|E'(\varphi_n)\|_{H^{-1}(\RNp)} \to 0.\]
\item[(b)] The proof of Proposition~\ref{boundednessPSCerami} is inspired by \cite[Section 3]{Je1999}. However, since our problem is not invariant under translations in $\RN$ and our nonlinearity $g$ has a non-standard shape, several difficulties appear. 
\end{itemize}
\end{remark}

\begin{proof}[Proof of Proposition~\ref{boundednessPSCerami}]
Let $d \in \R$ be an arbitrary but fixed constant and let $(u_n)_n \subset H_0^1(\RNp)$ be a Cerami sequence for $E$ at level $d \in \R$. First of all, observe that
\[ \|u_n^{-}\|^2 = - \langle E'(u_n), u_n^{-}\rangle \to 0, \quad \textup{ as } n \to \infty.\]
In particular, we deduce that $(u_n^{-})_n$ is bounded. It then remains to prove that $(u_n^{+})_n$ is bounded. We assume by contradiction that $\|u_n\| \to \infty$ and we set $v_n:= u_n\,/\,\|u_n\|$ for all $n \in \N$. Since $(v_n)_n$ and $(u_n^{-})_n$ are bounded, up to a subsequence if necessary, we have
\begin{equation}
v_n \weakto v \textup{ in } H_0^1(\RNp), \quad v_n \to v \textup{ in } L^q_{loc}(\RNp) \quad \text{for $1 \leq q < 2^{\ast}$} \quad \textup{ and } \quad v_n \to v \;\textup{a.e. in $\RNp$},
\end{equation}
for some $v \in H_0^1(\RNp)$ with $v \geq 0$. We have now two possible cases:
 
\begin{flushright}
\begin{minipage}[r]{0.9 \textwidth}
\textbf{\underline{Case 1} (Vanishing)}: For all $R > 0$, it follows that
\begin{equation} \label{vanishing}
\lim_{n \to \infty} \sup_{y \in \RNp} \int_{B_R(y)\cap \RNp} v_n^2 dx = 0.
\end{equation}
\noindent \textbf{\underline{Case 2} (Non-vanishing)}: There exist $R > 0,\ \delta > 0$ and a sequence of points $(y^n)_n \subset \RNp$ such that
\begin{equation} \label{non-vanishing}
\lim_{n \to \infty} \int_{B_R(y^n) \cap \RNp} v_n^2 dx \geq \delta.
\end{equation}
\end{minipage}
\end{flushright}
We shall prove that none of these cases may happen. This will prove the boundedness of the sequence Cerami sequence $(u_n)_n$. 

\medbreak 
 
\noindent\textbf{\underline{Case 1} (Vanishing)}: First of all, observe that, by \eqref{vanishing} and Lions' Lemma \cite[Lemma I.1]{LiII1984}, $v_n \to 0$ in $L^q(\RNp)$ for all $2 < q < 2^{\ast}$, and so, by uniqueness of the limit we have $v \equiv 0$. We define then the sequence $(z_n)_n \subset H_0^1(\RNp)$ by $z_n := t_n u_n$ with $t_n \in [0,1]$ satisfying 
$$ E(z_n) = \max_{t \in [0,1]} E(tu_n),$$
(if, for $n \in \N$, $t_n$ is not unique, we choose the smallest value) and we split the proof in the vanishing case (Case 1) into three steps.

\medbreak
\noindent \textbf{Step 1.1: } \textit{$\displaystyle \lim_{n \to \infty} E(z_n) = + \infty$.}
\medbreak
We argue by contradiction. Suppose there exists $M < +\infty$ such that
\[ \liminf_{n \to \infty} E(z_n) \leq M,\]
and define $(k_n)_n \subset H_0^1(\RNp)$ as
$$
k_n := \left(\frac{4M}{\widetilde q_c}\right)^{\frac{1}{2}}v_n = \left(\frac{4M}{\widetilde q_c}\right)^{\frac{1}{2}} \frac{1}{\|u_n\|} u_n,  \qquad \text{for all $n \in \N$},
$$
where $\widetilde q_c > 0$ is the constant given by Proposition \ref{positive-definite-main-proposition}. First, observe that
\begin{equation} \label{convergence-kn}
k_n \weakto 0 \textup{ in } H_0^1(\RNp), \quad k_n \to 0 \textup{ in } L^q(\RNp),\ \ \text{for $2 < q < 2^{\ast}$,} \quad \textup{ and } \quad k_n \to 0\textup{ a.e. in } \RNp.
\end{equation}
Then, by Corollary \ref{lemma-taylor} and \eqref{convergence-kn}, we obtain that
\begin{equation}
E(k_n) \geq \frac{\widetilde q_c}{2} \|k_n\|^2 - C\left(\|k_n\|_{L^{p+1}(\RNp)}^{p+1} + 1_{\{p>2\}} \|k_n\|_{L^3(\RNp)}^3  \right) = 2M + o(1).
\end{equation}
Taking $M$ bigger if necessary, we have that, for all $n \in \N$ large enough,
$$E(k_n) > \frac{3}{2}M.$$
On the other hand, observe that, for $n \in \N$ large enough, $\left(\frac{4M}{\widetilde q_c}\right)^{\frac{1}{2}} \frac{1}{\|u_n\|} \in [0,1].$ Hence, we have that
$$\frac{3}{2}M \leq \liminf_{n \to \infty} E(k_n) \leq \liminf_{n \to \infty} E(z_n) \leq M,$$
which is a contradiction. Thus, the Step 1.1 follows.
\medbreak
\noindent \textbf{Step 2.1:} $\langle E'(z_n), z_n \rangle = 0$ \textit{ for all $n \in \N$ large enough.} 
\medbreak
By Step 1.1 we know that $E(z_n) \to \infty$ as $n \to \infty$. On the other hand, $E(0) = 0$ and $E(u_n) \to d$ as $n \to \infty$. Hence, for $n \in \N$ large enough, $t_n \in (0,1)$ and so, by the definition of $z_n$, the Step 2.1 follows.
\medbreak
\noindent \textbf{Step 3.1:} \textit{Conclusion }\textbf{\underline{Case 1}}.
\medbreak
Observe that, by Step 2.1, for all $n \in \N$ large enough,
\begin{equation*}
E(z_n) = E(z_n) - \frac{1}{2}\langle E'(z_n),z_n \rangle = \int_{\RNp}H(x,z_n)dx,
\end{equation*}
where $H$ is given in Lemma \ref{G-g-estimate} (ii). By Step 1.1, we have that
\begin{equation} \label{contr-case-1}
\lim_{n \to \infty} \int_{\RNp}H(x,z_n)dx = +\infty.
\end{equation}
On the other hand, since $(u_n)_n$ is a Cerami sequence, 
\[ d+o(1) = E(u_n)-\frac{1}{2}\langle E'(u_n), u_n \rangle = \int_{\RNp} H(x,u_n) dx.\]
Then, using the definition of $z_n$ and the fact that $H(x,s)$ is non-decreasing in $s$ by Lemma~\ref{G-g-estimate} (ii), we obtain
\[ \int_{\RNp} H(x,z_n) dx \leq \int_{\RNp} H(x,u_n) dx = d+o(1),\]
which clearly contradicts \eqref{contr-case-1}. Hence, Case 1 (vanishing) cannot happen.

\bigbreak
\noindent \textbf{\underline{Case 2} (Non-vanishing)}: We split the proof into two steps.
\medbreak
\noindent \textbf{Step 1.2}: \textit{There exists $M > 0$ such that $y^n_{N}:= \dist(y^n,\partial \RNp) \leq M$ for all $n \in \N$}.
\medbreak
We assume by contradiction that $y^n_{N} \to +\infty$ as $n \to +\infty$. Then, for all $n \in \N$, we introduce $w_n:= v_n( \cdot + y^n)$ and observe that
\begin{equation}
w_n \weakto w \textup{ in } H^1(\RN), \quad w_n \to w \textup{ in } L^q_{loc}(\RN) \;\text{for $1 \leq q < 2^{\ast}$}, \quad \textup{ and } \quad w_n \to w \textup{ a.e. in } \RN,
\end{equation}
for some $w \in H^1(\RN)$ with $w \not \equiv 0$ (by \eqref{non-vanishing}) and $w \geq 0$. Now, observe that, since $(u_n)_n$ is a Cerami sequence, Lemma~\ref{G-g-estimate} (ii) implies that
\begin{align*}
o(1) & = \frac{1}{\|u_n\|^{p+1}}\Bigl(E(u_n) - \frac{1}{2} \langle E'(u_n),u_n \rangle\Bigr)  = \frac{1}{\|u_n\|^{p+1}} \int_{\RNp} H(x,u_n(x))\,dx\\
  &\ge \frac{1}{\|u_n\|^{p+1}}\Bigl[ \frac{p-1}{2(p+1)} \|u_n^+\|_{L^{p+1}(\RNp)}^{p+1}
  - \int_{\RNp} \bigl( u_c^{p-1}(x) D_{1,p}[u_n^+]^{2} + u_c(x) \, 1_{\{p>2\}} D_{2,p}[u_n^+]^{p}\bigr)dx\Bigr]\\
  &\ge \frac{1}{\|u_n\|^{p+1}}\Bigl[ \frac{p-1}{2(p+1)}\|u_n^+\|_{L^{p+1}(\RNp)}^{p+1}
  - \max\{ \woo^{p-1}, \woo\} \Bigl(D_{1,p}\|u_n^+\|_{L^{2}(\RNp)}^{2} +1_{\{p>2\}} D_{2,p}\|u_n^+\|_{L^p(\RNp)}^{p}\bigr)\Bigr]\\
  &\ge \frac{p-1}{2(p+1)}\|v_n^+\|_{L^{p+1}(\RNp)}^{p+1} -
 \frac{C}{\|u_n\|^{p+1}} \Bigl(\|u_n\|^{2} +1_{\{p>2\}}\|u_n\|^{p} \bigr)
\ge \frac{p-1}{2(p+1)} \|v_n^+\|_{L^{p+1}(\RNp)}^{p+1} +o(1),
\end{align*}
where $C>0$ is a constant independent of $n$. Here we also used Sobolev embeddings and the fact that $\|u_n\| \to \infty$ as $n \to \infty$. Since $p>1$, we thus conclude by Fatou's Lemma that 
$$
0 =\lim_{n \to \infty} \|v_n^+\|_{L^{p+1}(\RNp)}^{p+1}= \liminf_{n \to \infty} \int_{\{x_N \geq -y^n_{N}\}} (w_n^{+})^{p+1} dx \geq \int_{\RN} (w^{+})^p dx.
$$
Hence $w = w^+ \equiv 0$, which clearly is a contradiction. Thus, Step 1.2 follows.
\medbreak
\noindent \textbf{Step 2.2:} \textit{Conclusion }\textbf{\underline{Case 2}}.
\medbreak
By Step 1.2 we know there exists $M > 0$ such that $y^n_{N} \leq M$ for all $n \in \N$. We then define, for all $n \in \N$, $\widetilde{w}_n:= v_n( \cdot + \xi_n)$, where $\xi_n = (y^n_{1}, \ldots, y^n_{N-1},0)$. Again by \eqref{non-vanishing}, we have
\begin{equation}
\widetilde{w}_n \weakto \widetilde{w} \textup{ in } H_0^1(\RNp), \quad \widetilde{w}_n \to \widetilde{w} \textup{ in } L^q_{loc}(\RNp)\;  \text{for $1 \leq q < 2^{\ast}$}, \quad \textup{ and } \quad \widetilde{w}_n \to \widetilde{w} \textup{ a.e. in } \RNp,
\end{equation}
for some $\widetilde{w} \in H_0^1(\RNp)$ with $\widetilde{w} \not \equiv 0$ and $\widetilde{w} \geq 0$. For $n \in \N$, let $\varphi_n:= \widetilde{w}(\cdot - \xi_n) \in H^1_0(\RNp)$. Since $(u_n)_n$ is a Cerami sequence with $\|u_n\| \to \infty$ as $n \to \infty$, we have 
$$
\begin{aligned}
o(1) & =  \frac{\langle E'(u_n),\varphi_n \rangle}{\|u_n\|} = \frac{1}{\|u_n\|}\int_{\RNp} \bigl( \nabla u_n \nabla \varphi_n + u_n \varphi_n - g(x,u_n)\varphi_n \bigr)dx \\  
& = \int_{\RNp} \bigl( \nabla v_n \nabla \varphi_n + v_n \varphi_n\bigr)dx  - \int_{\RNp} \Bigl( \frac{(u_c+u_n^{+})^p-u_c^p}{\|u_n\|}\Bigr)\, \varphi_n dx\\
&=\int_{\RNp} \bigl( \nabla \widetilde{w}_n \nabla \widetilde{w} + \widetilde{w}_n \widetilde{w} \bigr)dx - \int_{\RNp} \Bigl( \frac{(u_c+\|u_n\| \widetilde{w}_n^{+})^p-u_c^p}{\|u_n\|}\Bigr)\, \widetilde{w} dx\\
&=\|\widetilde{w}\|^2 + o(1)- \int_{\RNp} \Bigl( \frac{(u_c+\|u_n\| \widetilde{w}_n^{+})^p-u_c^p}{\|u_n\|}\Bigr)\, \widetilde{w} dx.\\
\end{aligned}
$$
On the other hand, since $p>1$, we have that
\[
  \liminf_{n \to \infty} \frac{(u_c+\|u_n\|\widetilde{w}_n^{+})^p-u_c^p}{\|u_n\|\widetilde{w}_n^{+}} \, \widetilde{w}_n^{+}  \widetilde{w} = + \infty, \quad \textup{ a.e. in $\{\widetilde{w}>0\}$}
\]
and therefore, since $\widetilde{w} \ge 0$ and $\widetilde{w} \not \equiv 0$,
$$
\liminf_{n \to \infty} \int_{\RNp}  \Bigl(\frac{(u_c+u_n^{+}(\cdot + \xi_n))^p - u_c^p}{\|u_n\|}\Bigr)\,  \widetilde{w}dx= \liminf_{n \to \infty} \int_{\{\widetilde{w}>0\}}  \Bigl(\frac{(u_c+u_n^{+}(\cdot + \xi_n))^p - u_c^p}{\|u_n\|}\Bigr)\, \widetilde{w}dx = + \infty
$$
by Fatou's Lemma. This yields a contradiction. Hence, Case 2 (non-vanishing) cannot happen either and thus the result follows. 
\end{proof}

\begin{lem} \label{non-vanishing-lemma}
Let $(u_n)_n$ be a Cerami sequence for $E$ at level $d \in \R \setminus \{0\}$. Then, there exist $R > 0$, $\delta > 0$ and a sequence of points $(y^n)_n \subset \RNp$ such that
\[ \liminf_{n \to \infty} \int_{B_R(y^n) \cap \RNp} u_n^2\, dx \geq \delta.\]
\end{lem}

\begin{proof}
We assume by contradiction that, for all $R > 0$,
\[ \liminf_{n \to \infty} \sup_{y \in \RNp} \int_{B_R(y) \cap \RNp} u_n^2 dx = 0.\]
Then, by Lions' \cite[Lemma I.1]{LiII1984}, we have that $u_n \to 0$ in $L^q(\RNp)$ for all $2 < q < 2^{\ast}$. Now, since $(u_n)_n$ is a Cerami sequence, using Lemma~\ref{G-g-estimate} (i), we get
\begin{align*}
o(1) & = \langle E'(u_n), u_n \rangle = \|u_n\|^2 - \int_{\RNp}g(x,u_n(x))u_n\,dx \\
& \ge \|u_n\|^2 - \int_{\RNp} p u_c^{p-1} (u_n^{+})^2 dx - C_{1,p}\|u_n^+\|_{L^{p+1}(\RNp)}^{p+1} - 1_{\{p>2\}} C_{2,p} \|u_n^+\|_{L^3(\RNp)}^3 \\
& \ge q_c(u_n) - C_{1,p}\|u_n\|_{L^{p+1}(\RNp)}^{p+1} - 1_{\{p>2\}} C_{2,p} \|u_n\|_{L^3(\RNp)}^3 \\
& \geq \widetilde q_c \|u_n\|^2 - C_{1,p}\|u_n\|_{L^{p+1}(\RNp)}^{p+1} - 1_{\{p>2\}} C_{2,p} \|u_n\|_{L^3(\RNp)}^3.
\end{align*}
Hence, since $u_n \to 0$ in $L^{q}(\RNp)$ for all $2 < q < 2^*$, we deduce that $\|u_n\| \to 0$. Since $E$ is continuous, this implies that $E(u_n) \to 0$ as $n \to \infty$, contradicting our assumption that $d \not = 0$. The proof is finished.
\end{proof}

\section{Energy estimates} \label{section-energy-estimate}

\noindent We keep using the notation of the introduction and of Section~\ref{section-preliminaries}, which depends on the fixed quantities $c \in (0,\woo)$ and $p \in (1,2^{*}-1)$. Moreover, we will assume $N \ge 2$ throughout this section, which will be of key importance in order to derive the energy estimates we need. The mountain pass value associated to \eqref{non-autonomous-positive} is given by 
\begin{equation} \label{mpLevel}
b := \inf_{\gamma \in \Gamma} \max_{t \in [0,1]} E(\gamma(t)),
\end{equation}
where
$$ 
\Gamma:= \big\{ \gamma \in C([0,1], H_0^1(\RNp)): \gamma(0) = 0,\ E(\gamma(1)) < 0\, \big\}.
$$
We note that $b > 0$ by Lemma \ref{mpGeometry}. We also note that the functional $E$ (given in \eqref{E}) can be written as
\begin{equation}
  \label{eq:E-E-infty-relation}
E(u) = E_{\infty}^{+}(u) - \frac{1}{p+1} \int_{\RNp} \left( (u_c+u^{+})^{p+1}-u_c^{p+1}-(u^{+})^{p+1}-(p+1)u_c^p u^{+} \right) dx,
\end{equation}
where $E_{\infty}^{+}: H_0^1(\RNp) \to \R$ is given by
$$
E_{\infty}^{+}(u) = \frac{1}{2}\|u\|^2-\frac{1}{p+1} \int_{\RNp}(u^{+})^{p+1} dx.
$$

\medbreak
\noindent Now, we introduce the auxiliary (limit) problem
\begin{equation} \label{problemRn}
-\Delta u + u = |u|^{p-1}u, \quad u \in H^1(\RN),
\end{equation}
and its associated energy $E_{\infty}: H^1(\RN) \to \R$ given by 
\begin{equation} \label{Einfty}
E_{\infty}(u) = \frac{1}{2}\int_{\RN} \big( |\gradu|^2 + u^2\big)\,dx - \frac{1}{p+1} \int_{\RN} |u|^{\,p+1} dx.
\end{equation}
Also, we define
\begin{equation}
b_{\infty}:= \inf_{\mathcal{K}_{\infty}}E_{\infty}, \quad \text{ where } \quad \mathcal{K}_{\infty}:= \left\{ u \in H^1(\RN) \setminus \{0\}: E_{\infty}'(u) = 0 \right\}.
\end{equation}
According to \cite[Theorem 1]{BeLiI1983}, \cite[Th\'eor\`eme 1]{BeGaKa1983} and \cite[Theorem 2]{GiNiNi1981}, there exists a ground-state solution $\psi \in C^2(\RN)$ to \eqref{problemRn} which is positive, radially symmetric, and such that
\begin{equation} \label{behaviour-gs-infinity}
\psi(x) \leq C_{GS} |x|^{-\frac{N-1}{2}} e^{-|x|} \quad \textup{ and } \quad |\nabla \psi(x)| \leq C_{GS} |x|^{-\frac{N-1}{2}} e^{-|x|}, \qquad \textup{ as} \quad |x| \to \infty,
\end{equation}
for some $C_{GS} > 0$ depending only on $N$ and $p$. Moreover,
\begin{equation}
  \label{eq:strict-decrease}
\text{$\psi$ is strictly decreasing in the radial variable.}
\end{equation}
Let us also emphasize that
\begin{equation} \label{gs-energy-identity}
0 < b_{\infty} = E_\infty(\psi)= \left(\frac{1}{2}-\frac{1}{p+1} \right) \|\psi\|_{L^{p+1}(\RN)}^{p+1} = \frac{p-1}{2(p+1)} \|\psi\|_{L^{p+1}(\RN)}^{p+1}=  \frac{p-1}{2(p+1)} \inf_{u \in {\cK}_{\infty}}\|u\|_{L^{p+1}(\RN)}^{p+1}.
\end{equation}

\bigbreak
The aim of this section is to show, based on the assumption $N \ge 2$, that
\begin{equation} \label{b-smaller-binfty}
b < b_{\infty}.
\end{equation}
This strict inequality will be crucial to prove the existence result to \eqref{non-autonomous-positive} contained in Section \ref{section-existence}. To this end, let us recall that $u_c(x) \sim e^{-x_N}$ as $x_N \to \infty$. More precisely, it follows from (\ref{eq:def-w-0}) and the definition of $u_c$ that
\begin{equation} \label{behaviour-u0-infinity}
  m_{c,1}\, e^{-x_N} \le u_c(x) \le m_{c,2}\, e^{-x_N}, \quad \text{for $x \in \RNp, \quad $ with}\qquad m_{c,1} :=  \woo  e^{-t_{c,p}}, \quad m_{c,2} :=  \woo 2^{\frac{2}{p-1}}  e^{-t_{c,p}}.
\end{equation}
Moreover, for $r > 0$, we introduce the function
$$
x \mapsto \psi_r(x):= (\psi(x- r e_N) - \epsilon_r)^{+},
$$
where $e_N = (0, \ldots,0,1)$ is the $N$-th coordinate vector and $\epsilon_r>0$ is uniquely defined by (\ref{eq:strict-decrease}) and the property that
$$
\psi > \epsilon_r \qquad \text{in $B_r(0)$}\qquad \text{and}\qquad \psi \le \eps_r \qquad \text{in $\R^N \setminus B_r(0)$.}
$$
We note that, as a consequence of \eqref{behaviour-gs-infinity}, we have
\begin{equation}
  \label{eq:est-epsilon-r}
\eps_r \le  C_{GS} r^{-\frac{N-1}{2}} e^{-r}
\end{equation}
We also note that $\psi_r \in H_0^1(\RNp)$ for every $r \ge 0$.

\medbreak
The rest of the section is devoted to prove the following result from which \eqref{b-smaller-binfty} immediately follows.

\begin{prop} \label{main-proposition-energy-estimate}$ $
There exists $R>0$ and $k>0$ with the following properties: 
  \begin{itemize}
\item[i)]  $E(t\psi_r) < b_{\infty}$ for all $r \geq R,\: t \in [0,k]$.
\item[ii)] $E(k \psi_r) < 0$ for all $r \geq R$.
\end{itemize}
\end{prop}

\medbreak
We split the proof of this proposition into several lemmas.

\begin{lem} \label{lemma1-energy-estimate}
There exists $C_1 > 0$ with 
\begin{equation} \label{step1-energy-estimate}
E_{\infty}^{+} (t\psi_r) \leq b_{\infty} + C_1 e^{-r} r^{-\frac{N-1}{2}} t^{p+1} \qquad \text{for all $t,r >0$. }
\end{equation}
\end{lem}

\begin{proof}
Let $t,r>0$. Directly observe that, by the definition of $\psi_r$,
\begin{equation}
  \label{step1-eq1}
E_{\infty}^{+}(t\psi_r) = \frac{t^2}{2} \|\psi_r\|^2 - \frac{t^{p+1}}{p+1} \int_{\RNp} \psi_r^{p+1} dx = \frac{t^2}{2} \int_{B_r(0)} \big( |\nabla\psi|^2 + (\psi - \epsilon_r)^2 \big)\, dx - \frac{t^{p+1}}{p+1} \int_{B_r(0)} (\psi- \epsilon_r)^{p+1} dx. 
\end{equation}
On the other hand, since $\psi$ is a solution to \eqref{problemRn} and  $(\psi-\epsilon_r)^{+} \in H^1(\RN)$, using $(\psi-\epsilon_r)^{+}$ as test function in \eqref{problemRn}, we obtain that 
\begin{equation*}
\begin{aligned}
\int_{B_r(0)} |\gradpsi|^2 dx & = \int_{\RN} \gradpsi \nabla (\psi-\epsilon_r)^{+} dx \\ 
& = \int_{\RN} \left( - \psi (\psi - \epsilon_r)^{+} + \psi^p (\psi - \epsilon_r)^{+} \right) dx =  \int_{B_r(0)} \left( - \psi (\psi - \epsilon_r) + \psi^p (\psi - \epsilon_r) \right) dx.
\end{aligned}
\end{equation*}
Substituting the above identity into \eqref{step1-eq1} and using the mean value theorem, we find that
\begin{align*}
E_{\infty}^{+}(t\psi_r)&  = \frac{t^2}{2} \int_{B_r(0)} \left( -\psi(\psi-\epsilon_r) + \psi^p (\psi-\epsilon_r) + (\psi-\epsilon_r)^2 \right) dx - \frac{t^{p+1}}{p+1} \int_{B_r(0)} (\psi-\epsilon_r)^{p+1} dx \\
& = - \frac{\epsilon_r t^2}{2} \int_{B_r(0)} (\psi-\epsilon_r) dx +
\frac{1}{p+1} \int_{B_r(0)} (\psi-\epsilon_r) \left[ \frac{p+1}{2} t^2 \psi^p - t^{p+1} (\psi-\epsilon_r)^p \right] dx \\
& \leq \frac{1}{p+1} \int_{B_r(0)} (\psi-\epsilon_r) \left[ \frac{p+1}{2} t^2 \psi^p - t^{p+1} (\psi-\epsilon_r)^p \right] dx \\
& = \frac{1}{p+1} \left(\frac{p+1}{2}t^2-t^{p+1} \right) \int_{B_r(0)} (\psi-\epsilon_r) \psi^p dx + \frac{t^{p+1}}{p+1} \int_{B_r(0)} (\psi-\epsilon_r) \left[ \psi^p - (\psi-\epsilon_r)^p \right] dx\\
& \leq \frac{p-1}{2(p+1)} \int_{B_r(0)} (\psi-\epsilon_r) \psi^p dx
+ \frac{\epsilon_r\,p\, t^{\,p+1}}{p+1} \int_{B_r(0)} (\psi-\epsilon_r) \psi^{p-1} dx  \\
& \leq \frac{p-1}{2(p+1)} \int_{\RN} \psi^{\,p+1} dx + \frac{\epsilon_r\,p\, t^{\,p+1}}{p+1} \int_{\RN} (\psi-\epsilon_r)^{+} \psi^{p-1} dx.
\end{align*}
Using \eqref{gs-energy-identity} and (\ref{eq:est-epsilon-r}), we deduce that
$$
E_{\infty}^{+}(t\psi_r)  \leq b_{\infty} + \frac{\epsilon_r\,p\, t^{\,p+1}}{p+1} \int_{\RN} (\psi-\epsilon_r)^{+} \psi^{p-1} dx  \leq b_{\infty} +  \frac{p}{p+1}C_{GS} r^{-\frac{N-1}{2}} e^{-r} t^{p+1} \int_{\RN} \psi^p dx.  
$$
Hence (\ref{step1-energy-estimate}) holds with $C_1 = \frac{p C_{GS}}{p+1} \int_{\RN} \psi^p dx$.
\end{proof}

\begin{lem} \label{lemma2-energy-estimate}
There exists ${R'}\ge 1$ and $C_2 > 0$ with 
\begin{equation} \label{step2-energy-estimate}
E_{\infty}^{+}(t\psi_r) - E(t \psi_r) \geq C_2 e^{-r}\, t^{\,p}  \qquad \text{for all $t >0$, $r \geq {R'}$.}
\end{equation}
\end{lem}

\begin{proof}
  Let $t >0$ be arbitrary but fixed. Using Lemma \ref{inequality-q-refined} with $q = p+1$, $\kappa:= \kappa_{q}>0$, the identity (\ref{eq:E-E-infty-relation}), the lower bound in (\ref{behaviour-u0-infinity}) and the mean value theorem, we deduce that, for all $r \geq 1$,
\begin{align*}
E_{\infty}^{+}(t\psi_r)  - E(t\psi_r)  & = \frac{1}{p+1} \int_{\RNp} \left(  (u_c+t\psi_r)^{p+1}-u_c^{p+1} - (t\psi_r)^{p+1} -(p+1)u_c^p t \psi_r  \right) dx  \geq \kappa t^p \int_{\RNp} u_c \psi_r^p dx\\
  &= \kappa t^p \int_{\RNp} u_c\bigl((\psi(\cdot-re_N) - \epsilon_r)^{+}\bigr)^p dx = \kappa t^p \int_{B_r(0)} u_c(\cdot+re_N) (\psi - \epsilon_r)^p dx\\
  &\geq \kappa t^p  \int_{B_r(0)} u_c(\cdot+re_N)\bigl( \psi^p  -\epsilon_r p \psi^{p-1}\bigr)\, dx  \geq \kappa t^p \Bigl( \int_{B_r^-(0)} u_c(\cdot+re_N ) \psi^p dx - \epsilon_r p c_p \int_{\RN} \psi^{p-1} dx\Bigr)\\
  &\ge \kappa t^p \Bigl(m_{c,1}e^{-r}  \int_{B_r^{\,-}(0)} \psi^p dx -  p c_p C_{GS} r^{-\frac{N-1}{2}} e^{-r} \int_{\RN} \psi^{p-1} dx\Bigr)\\
  &\ge \kappa t^p  e^{-r} \Bigl(\frac{m_{c,1}}{2} \int_{B_r(0)} \psi^p dx - p c_p C_{GS} r^{-\frac{N-1}{2}} \int_{\RN} \psi^{p-1} dx\Bigr) \\
 & \ge \kappa t^p  e^{-r} \Bigl(\frac{m_{c,1}}{2} \int_{B_1(0)} \psi^p dx - p c_p C_{GS} r^{-\frac{N-1}{2}} \int_{\RN} \psi^{p-1} dx\Bigr)
\end{align*}
where we have set $B_r^{\,-}(0):= \{x \in B_r(0)\::\: x_N \le 0\}$. Since $N \ge 2$, we may choose ${R'} \ge 1$ sufficiently large to guarantee that 
$$
p c_p C_{GS} r^{-\frac{N-1}{2}} \int_{\RN} \psi^{p-1} dx \le \frac{m_{c,1}}{4} \int_{B_1(0)} \psi^p dx, \qquad \text{for $r \ge {R'}$},
$$
and therefore
\begin{equation*}
  E_{\infty}^{+}(t\psi_r) - E(t\psi_r)  \geq C_2 e^{-r}\, t^{\,p}, \quad \text{for $r \ge {R'}$, \quad with}\quad
C_2:= \frac{\kappa m_{c,1}}{4} \int_{B_1(0)} \psi^p dx.
\end{equation*}
Hence the claim follows.
\end{proof}

\begin{lem} \label{lemma3-energy-estimates}
Let ${R'}\ge 1$ be given as in Lemma~\ref{lemma2-energy-estimate}. Then there exist $k > 0$ with $E(k\psi_r) < 0$ for all $r \ge {R'}$.
\end{lem}

\begin{proof}
  Let $k >0$. For $r \ge {R'} \ge 1$ we have, by Lemma~\ref{lemma2-energy-estimate} and since the map $r \mapsto \epsilon_r$ is strictly decreasing by
(\ref{eq:strict-decrease}),  
\begin{align*}
  E(k \psi_r) & \leq E_{\infty}^+(k\psi_r)  = \frac{k^2}{2} \|\psi_r\|^2 - \frac{k^{p+1}}{p+1} \int_{\RNp} \psi_r^{p+1} dx \leq \frac{k^2}{2} \|\psi\|_{H^1(\RN)}^2 - \frac{k^{p+1}}{p+1} \int_{B_r(0)} \big(\psi - \epsilon_r \big)^{p+1} dx \\
  &\leq \frac{k^2}{2} \|\psi\|_{H^1(\RN)}^2 - \frac{k^{p+1}}{p+1} \int_{B_1(0)} \big(\psi - \epsilon_1 \big)^{p+1} dx 
\end{align*}
Since $\int_{B_1(0)} \big(\psi - \epsilon_1 \big)^{p+1} dx >0$ by (\ref{eq:strict-decrease}), we may choose
$$
k> \left( \frac{(p+1)\|\psi\|_{H^1(\RN)}^2}{2 \int_{B_1(0)} \big(\psi - \epsilon_1 \big)^{p+1} dx}\right)^{\frac{1}{p-1}}
$$
which implies that $ E(k \psi_r) < 0$ for $r \ge {R'}$, as claimed. 
\end{proof}

\begin{proof}[\textbf{Proof of Proposition \ref{main-proposition-energy-estimate}}]
  Let ${R'}\ge 1$ be given by Lemma~\ref{lemma2-energy-estimate}, and let $k>0$ be given by Lemma~\ref{lemma3-energy-estimates}. For $r \ge {R'}$ and $t \in [0,k]$ we then have, by Lemmas \ref{lemma1-energy-estimate} and \ref{lemma2-energy-estimate},
$$
E(t \psi_r)  = E_{\infty}^{+}(t \psi_r) - (E_{\infty}^{+}(t\psi_r)-E(t\psi_r))  \leq b_{\infty} + e^{-r}  \bigl( t^{p+1} C_1 r^{-\frac{N-1}{2}} - t^p C_2 \bigr) 
\le b_{\infty} + e^{-r} t^p \bigl (k C_1 r^{-\frac{N-1}{2}} - C_2 \bigr). 
$$
Since $N \ge 2$, we may fix $R \ge R'$ with the property that $k C_1 r^{-\frac{N-1}{2}} \le \frac{C_2}{2}$ for $r \ge R$, which implies that 
$$
E(t \psi_r) \le b_{\infty} - \frac{C_2}{2} e^{-r} t^p < b_\infty \quad \text{for $t \in (0,k]$, $r \ge R$.}
$$
Since also $E(0)= 0 < b_{\infty}$, we thus obtain that $E(t \psi_r) < b_{\infty}$ for $t \in [0,k]$, $r \ge R$. Moreover, by Lemma~\ref{lemma3-energy-estimates} we have $E(k\psi_r) < 0$ for all $r \ge R$ since $R \ge R'$. The proof is finished.
\end{proof}

\section{The existence result} \label{section-existence}

\noindent We keep using the notation of the introduction and of Section~\ref{section-preliminaries}, which depends on the fixed quantities $c \in (0,\woo)$ and $p \in (1,2^{*}-1)$. Moreover, we will assume $N \ge 2$ throughout this section, which will allow us to prove the existence of a non-trivial solution to \eqref{non-autonomous-positive}. This will conclude the proof of Theorem~\ref{th1-existence-refined}.

\begin{thm} \label{existence-non-autonomous}
Let $N \geq 2$. Then there exists a non-trivial solution $u \in H_0^1(\RNp)$ to \eqref{non-autonomous-positive} which, in particular, is a positive solution to \eqref{non-autonomous}.
\end{thm}

The strategy of the proof is as follows: using the strict inequality \eqref{b-smaller-binfty}, we will manage to prove the existence of a  Cerami sequence whose weak limit is non trivial and thus we will obtain a non trivial solution to \eqref{non-autonomous-positive}.

\begin{proof}[Proof of Theorem \ref{existence-non-autonomous}]
  Since the functional $E$ has a mountain pass geometry (see Lemma \ref{mpGeometry}), there exists a Cerami sequence for $E$ at the corresponding mountain pass level $b$ defined in (\ref{mpLevel}) (see e.g. \cite{Ce1978} or \cite[Theorem 6, Section 1, Chapter IV]{Ek1990}), i.e. there exists $(u_n)_n \subset H_0^1(\RNp)$ such that
  \[ E(u_n) \to b \quad \textup{ and } \quad (1+\|u_n\|)\|E'(u_n)\|_{H^{-1}(\RNp)} \to 0, \quad \textup{ as } n \to \infty.\]
  By Proposition~\ref{boundednessPSCerami} we know that $(u_n)_n$ is bounded in $H_0^1(\RNp)$. Moreover,
  \begin{equation}
    \label{eq:un-minus-sec-5}
  \|u_n^-\|^2 = \langle E'(u_n),u_n^- \rangle \to 0 \qquad \text{as $n \to \infty$.}
  \end{equation}
%  Hence, up to a subsequence, we have that
%\begin{equation} \label{congergence-existence-result}
%u_n \weakto u \textup{ in } H_0^1(\RNp), \quad u_n \to u \textup{ in } L_{loc}^%q(\RNp) \textup{ for } 1 \leq q < 2^{\ast} \quad \textup{ and }  \quad u_n \to %u\; \text{ a.e. in } \RNp, 
%\end{equation}
%for some $u \in H_0^1(\RNp)$ with $u_n \ge 0$. If we show that $u \not \equiv 0%$, then $u$ is a non-trivial solution to \eqref{non-autonomous-positive} and th%e proof is finished.

  Let $(y^n)_n \subset \RNp$ be the sequence of points obtained in Lemma \ref{non-vanishing-lemma} applied to $(u_n)_n$, i.e., we have
  \begin{equation}
    \label{eq:lemma-3-3-consequence}
 \liminf_{n \to \infty} \int_{B_R(y^n) \cap \RNp} u_n^2\, dx \geq \delta \qquad \text{for some $\delta>0$.}
   \end{equation}
We split the argument into two steps.

\medbreak
\noindent \textbf{Step 1: } \textit{There exists $M > 0$ such that $y^n_{N} = \dist(y^n, \partial \RNp) \leq M$ for all $n \in \N$.}
\medbreak
We assume by contradiction that
\begin{equation} \label{y_n-to-infty}
\displaystyle \lim_{n \to \infty} y^n_{N} = + \infty.
\end{equation}
Then, let us define, for all $n \in \N$, $w_n := u_n (\cdot + y^n)$.
By Lemma \ref{non-vanishing-lemma} and (\ref{eq:un-minus-sec-5}), it follows that 
\begin{equation}
w_n \weakto w \textup{ in } H^1(\RN), \quad w_n \to w \textup{ in } L^q_{loc}(\RN)\ \;\text{for $1 \leq q < 2^{\ast}$}, \quad \textup{ and } \quad w_n \to w \textup{ a.e. in } \RN,
\end{equation}
for some $w \in H^1(\RN)$ with $w \ge 0$, $w \not \equiv 0$. We also observe that 
\begin{equation}
  \label{fatou-final-prelim}
  b + o(1)  = E(u_n) - \frac{1}{2} \langle E'(u_n), u_n \rangle = \int_{\RNp} H(x,u_n(x))\,dx=\int_{ \{x_N  \geq -y^n_{N} \}} H(x+y_n,w_n^{+}(x)) dx, \qquad \text{as $n \to \infty$}, 
\end{equation}
with the function $H$ defined in Lemma~\ref{G-g-estimate} (ii).
Next, we note that
$$
H(x+y_n,w_n^{+}(x)) \ge 0, \qquad \text{for $x \in \{x_N  \geq -y^n_{N} \}$},
$$
and
\begin{align*}
\liminf_{n \to \infty} H(x+y_n,w_n^{+}(x)) &\ge \liminf_{n \to \infty}\,\Bigl(\frac{p-1}{2(p+1)}\,[w_n^+(x)]^{p+1} - u_c(x+y_n)^{p-1} D_{1,p} [w_n^+(x)]^2 - u_c(x+y_n) 1_{\{p > 2\}} D_{2,p} [w_n^+(x)]^p \Bigr)\\
&= \frac{p-1}{2(p+1)}\,[w^+(x)]^{p+1}= \frac{p-1}{2(p+1)}\,w^{p+1}(x), \qquad \text{for $x \in \{x_N  \geq -y^n_{N} \}$},
\end{align*}
by Lemma~\ref{G-g-estimate} (ii) and \eqref{behaviour-u0-infinity}. Thus, (\ref{fatou-final-prelim}) and Fatou's Lemma imply that 
\begin{equation}
  \label{eq:less-equal-b}
\frac{p-1}{2(p+1)}\|w\|_{L^{p+1}(\R^N)}^{p+1} \le b.
\end{equation}
Next we claim that $w \in \cK_\infty$, i.e., $w$ is a nontrivial solution of \eqref{problemRn}. To see this, we fix an arbitrary $\varphi \in C_c^{\infty}(\RN)$, and we show that 
\begin{equation}
  \label{eq:distr-solut}
\int_{\RN} \big(\nabla w \nabla \varphi + w \varphi \big) dx = \int_{\RN} (w^{+})^{\,p} \varphi\, dx.
\end{equation}
Since \eqref{y_n-to-infty} holds, we have that $\supp(\varphi) \subset \{x_N \geq -y^n_{N}\}$ for $n \in \N$ sufficiently large. Hence, for $n\in \N$ large enough, we have that
\begin{equation}
\begin{aligned}
o(1) & = \langle E'(u_n), \varphi (\cdot - y^n) \rangle\\
& = \int_{\{x_N \geq - y^n_{N}\}} \left(\nabla w_n \nabla \varphi + w_n \varphi \right) dx - \int_{\{x_N \geq - y^n_{N}\}} \big(u_c(\cdot+y^n)+w_n^{+}\big)^p \varphi dx - \int_{\{x_N \geq - y^n_{N}\}} u_c^p(\cdot+y^n) \varphi dx \\
& = \int_{\RN} \left(\nabla w_n \nabla \varphi + w_n \varphi \right) dx - \int_{\RN} \big(u_c(\cdot+y^n)+w_n^{+}\big)^p \varphi dx + o(1) \\
& =  \int_{\RN} \left(\nabla w_n \nabla \varphi + w_n \varphi \right) dx - \int_{\RN} (w_n^{+})^p \varphi dx - \int_{\RN} \big( (u_c(\cdot+y^n)+w_n^{+})^p - (w_n^{+})^p \big) \varphi dx + o(1) \\
& = \int_{\RN} \left(\nabla w_n \nabla \varphi + w_n \varphi \right) dx - \int_{\RN} (w_n^{+})^p \varphi dx + o(1)  \\
& = \int_{\RN} \big(\nabla w \nabla \varphi + w \varphi \big) dx - \int_{\RN} (w^{+})^{\,p} \varphi\, dx  + o(1), \qquad \text{as $n \to \infty$.}
\end{aligned}
\end{equation}
Hence (\ref{eq:distr-solut}) follows, and therefore $w \in \cK_\infty$. Together with (\ref{gs-energy-identity}) and (\ref{eq:less-equal-b}) it then follows that $b \geq b_{\infty}$, but this contradicts \eqref{b-smaller-binfty}. Hence, \eqref{y_n-to-infty} cannot happen and Step 1 follows.

\medbreak

\noindent \textbf{Step 2: }\textit{ Conclusion.}

\medbreak
Let us define, for all $n \in \N$, $v_n:= u_n( \cdot + \xi_n)$ with $\xi _n:= (y^n_{1}, \ldots, y^n_{{N-1}}, 0)$ and observe that, after passing to a subsequence
\begin{equation}
v_n \weakto v \textup{ in } H_0^1(\RNp), \quad v_n \to v \textup{ in } L_{loc}^q(\RNp) \textup{ for } 1 \leq q < 2^{\ast} \quad \textup{ and }  \quad v_n \to v \textup{ a.e. in } \RNp, 
\end{equation}
for some $v \in H_0^1(\RNp)$. Also, note that  $(v_n)_n \subset H_0^1(\RNp)$ is a Cerami sequence for $E$ at level $b$. Hence, if $v \not \equiv 0$, we will have that $v$ is a non-trivial solution to \eqref{non-autonomous-positive}. Since $v_n \to v$ in $L^q_{loc}(\RNp)$ and $y^n_{N} \leq M$ for all $n \in \N$, the lower integral bound (\ref{eq:lemma-3-3-consequence}) implies that $v \not \equiv 0$, and the result follows. 
\end{proof}

\begin{proof}[\textbf{Proof of Theorem \ref{th1-existence-refined} (completed)}]
Let $u \in H_0^1(\RNp)$ be the non-negative and non-trivial solution to \eqref{non-autonomous} obtained in Theorem \ref{existence-non-autonomous}. By standard elliptic regularity we have that $u \in C^2(\RNp) \cap C(\overline{\RNp}) \cap L^{\infty}(\RNp)$ and  $v := u_c +  u$ is a bounded positive solution to \eqref{mainProblem}-\eqref{cond-infinity} of the form (\ref{eq:solution-ansatz}).  
\end{proof}

\begin{remark}
As explained in the introduction, Theorem~\ref{th1} (i) and Corollary \ref{th1-existence-refined-multiplicity} are direct consequences of Theorem \ref{th1-existence-refined}. 
\end{remark}

\medbreak
In the remaining of this section we prove Proposition \ref{boundedness-of-the-solutions}. Let us first state a technical lemma due to Pol\'{a}\v{c}ik, Quittner and Souplet that will be key to prove this result.

\begin{lem}{\rm\textbf{(Particular case of \cite[Lemma 5.1]{PoQuiSou2007})} } \label{lemma-PoQuiSou}
 \it Let $(X,d)$ be a complete metric space and let $M: X \to (0,+\infty)$ be continuous. For any $\delta < \sup_{X} M$ and any $k > 0$ there exists $y \in X$ such that
\begin{itemize}
\item[$\bullet$] $M(y) \geq \delta$.
\item[$\bullet$] $M(z) \leq 2 M(y)$ for all $z \in X$ with $d(z,y) \leq \frac{k}{M(y)}$.
\end{itemize}
\end{lem}

\medbreak
The following proof is inspired by \cite[Lemma 2.5]{DaWe2010}.

\begin{proof}[\textbf{Proof of Proposition \ref{boundedness-of-the-solutions}}]
We assume by contradiction that there exists $v \in C^2(\RNp) \cap C(\overline{\RNp})$ unbounded solving \eqref{mainProblem}-\eqref{cond-infinity}. By Lemma \ref{lemma-PoQuiSou} applied with $X= \overline{\RNp}$ and $M = v^{\frac{p-1}{2}}$, there exits a sequence $(y^k)_k \subset \RNp$ such that
\begin{align} \label{consequence1-lemma-PoQuiSou}
& M(y^k) \to \infty, \quad \textup{ as }\ k \to \infty, \\
& M(z) \leq 2 M(y^k), \quad \textup{ for all } z \in \RNp\ \textup{ with }\ d(z,y^k) \leq \frac{k}{M(y^k)}\ \textup{ and all }\ k \in \N. \label{consequence2-lemma-PoQuiSou}
\end{align}
Note that, without loss of generality, we can suppose that $M(y^k) \geq 1$ for all $k \in \N$. We then define, for all $k \in \N$, $d_k:= y_N^k M(y^k)$, the half-space $H_k := \big\{x \in \RNp: x_N > - d_k \big\}$ and 
$$
v^k: H_k \to \RNp \quad \textup{ given by }\quad v^k(z):= \frac{1}{v(y^k)} v\big( y^k + \frac{z}{M(y^k)} \big)\,.
$$
Note that, for all $k \in \N$, $v^k$ is a positive solution to
\begin{equation} \label{vk-equation-boundedness}
\left\{
\begin{aligned}
-\Delta v^k + \frac{1}{M^2(y^k)} v^k & = (v^k)^p, & \quad \textup{ in } H_k,\\
v^k & = \frac{c}{M^{\frac{2}{p-1}}(y^k)}, & \textup{ on } \partial H_k,
\end{aligned}
\right.
\end{equation}
and, by its definition and \eqref{consequence2-lemma-PoQuiSou}, it satisfies 
\begin{equation} \label{vk-properties-boundedness}
v^k(0) = 1 \quad \textup{ and } \quad v^k(z) \leq 2^{\frac{2}{q-1}} \quad \textup{ for all } z \in H_k \cap B_k(0).
\end{equation}
We now consider two cases separately.
\medbreak
\noindent \textit{Case 1: } $d_k \to \infty\ $ as $\ k \to \infty$.
\medbreak
Using standard $L^q$ estimates (see e.g. \cite[Chapter 9]{G_T_2001_S_Ed}), \eqref{vk-equation-boundedness} and \eqref{vk-properties-boundedness}, we get (taking a subsequence if necessary) that $(v^k)_k$ is locally $W^{2,q}$-bounded in $\RN$ for arbitrarily large $q < +\infty$. Hence, up to a subsequence, $v^k \to \overline{v}$ in $C_{loc}^{1}(\RN)$, where $\overline{v} \in C^1(\RN)$ is a non-trivial positive solution to
\begin{equation}  \label{Gispr-1}
\Delta \overline{v} + \overline{v}^p = 0, \quad \textup{ in } \RN.
\end{equation}
By \eqref{vk-properties-boundedness} we infer that $\overline{v}$ is bounded in $\RN$ with $\overline{v}(0)=1$. Hence $\overline{v} \in C^2(\RN)$ by standard elliptic regularity. Then, since by \cite[Theorem 1.2]{GiSpr1981-2} we know that the only $C^2(\RN)$ non-negative solution to \eqref{Gispr-1} is $\overline{v} \equiv 0$, we obtain a contradiction and deduce that \textit{ Case 1} cannot happen. 

\medbreak
\noindent \textit{Case 2: } $d_k \to d \geq 0\ $ as $\ k \to \infty$.
\medbreak
Let us define, for all $k \in \N$,
\begin{equation}
w^k: \RNp \to \R \quad \textup{ as } \quad w^k(z) = v^k (z-d_k e_N),
\end{equation}
where $e_N := (0, \ldots,0,1)$ is the $N$-th coordinate vector. Note that, for all $k \in \N$, $w^k$ is a positive solution to 
\begin{equation} \label{wk-equation-boundedness}
\left\{
\begin{aligned}
-\Delta w^k + \frac{1}{M^2(y^k)} w^k & = (w^k)^p, & \quad \textup{ in } \RNp,\\
w^k & = \frac{c}{M^{\frac{2}{p-1}}(y^k)}, & \textup{ on } \partial \RNp,
\end{aligned}
\right.
\end{equation}
and satisfies 
\begin{equation} \label{wk-properties-boundedness}
w^k(d_k e_N) = 1 \quad \textup{ and } \quad w^k(z) \leq 2^{\frac{2}{p-1}} \quad \textup{ for all } z \in \overline{\RNp} \cap B_k(d_k e_N).
\end{equation}
Now, arguing as in the proof of \cite[Theorem 9.13]{G_T_2001_S_Ed} (with auxiliary functions $\varphi^k = w^k - c\, M^{-\frac{2}{p-1}}(y^k)$) and taking into account \eqref{wk-equation-boundedness} and \eqref{wk-properties-boundedness}, we get (taking a subsequence if necessary) that $(w^k)_k$ is locally $W^{2,q}$-bounded in $\overline{\RNp}$ for arbitrarily large $q < +\infty$ and therefore also locally $C^{1,\beta}$-bounded in $\overline{\RNp}$ for all $\beta \in (0,1)$. In particular, $|\nabla w^k|$ remains bounded pointwise independently of $k$ in a neighbourhood of the origin. Taking into account \eqref{consequence1-lemma-PoQuiSou}, the boundary conditions in \eqref{wk-equation-boundedness} and \eqref{wk-properties-boundedness}, we infer that $d = \lim_{k \to \infty} d_k > 0$. Hence, up to a subsequence, $w^k \to \overline{w}$ in $C^{1}_{loc}(\RNp)$ with $\overline{w} \in C^1(\RNp)$ a non-trivial positive solution to 
\begin{equation} \label{Gispr-2}
\left\{
\begin{aligned}
\Delta \overline{w} + \overline{w}^p & = 0, & \quad \textup{ in } \RNp,\\
\overline{w} & = 0, & \quad \textup{ on } \partial \RNp.
\end{aligned}
\right.
\end{equation}
By \eqref{wk-properties-boundedness} we have that $\overline{w}$ is bounded with $\overline{w}(d e_N)=1$. Hence $\overline{w} \in C^2(\RN) \cap C(\overline{\RNp})$ by standard elliptic regularity. Then, since by \cite[Theorem 1.3]{GiSpr1981-2} we know that the only $C^2(\RNp) \cap C(\overline{\RNp})$ non-negative solution to \eqref{Gispr-2} is $\overline{w} \equiv 0$ , we obtain a contradiction and deduce that \textit{Case 2} cannot happen either. Hence, the result follows. 
\end{proof}

\section{The non-existence result} \label{section-non-existence-uniqueness}

\noindent In this section we prove Part (ii) of Theorem~\ref{th1}, which is concerned with the non-existence of bounded positive solutions to \eqref{mainProblem}-\eqref{cond-infinity} in the case $c > \woo$. Recall that
$$
w_0(t) = \woo \left[ \cosh \left( \frac{p-1}{2} t \right) \right]^{-\frac{2}{p-1}},
$$
is the unique even non-trivial positive solution to \eqref{1dim}. Throughout this section, we will use the following notation. We define $v_0: \RN \to \R$ as 
\begin{equation} \label{v0}
v_0(x) = w_0(x_N), \qquad \text{for $x \in \RN.$}
\end{equation}
Also, recall that for a bounded positive solution to \eqref{mainProblem}-\eqref{cond-infinity}, we mean a function $v \in C^2(\RNp) \cap C(\overline{\RNp}) \cap L^{\infty}(\RNp)$, positive, satisfying \eqref{mainProblem} in the pointwise sense and such that \eqref{cond-infinity} holds.

\begin{thm} \label{non-existence}
For $N \geq 1$, $p > 1$ and $c > \woo$, there are no bounded positive solutions to \eqref{mainProblem}-\eqref{cond-infinity}.
\end{thm}

\begin{proof}
Let us fix an arbitrary $c > \woo$. We assume by contradiction that there exists a bounded positive solution $v$ to \eqref{mainProblem}-\eqref{cond-infinity} and we define, for all $t \in \R$, $v_t:= v_0 (\cdot + te_N)$ where $v_0$ is given in \eqref{v0} and $e_N = (0,\ldots,0,1)$ is the $N$-th coordinate vector. We split the proof into three steps.

\medbreak
\noindent \textbf{Step 1: } \textit{There exists $t_0 > 0$ such that $v > v_t$ in $\overline{\RNp}$ for all $t \geq t_0$.}

\medbreak
First of all, fixed an arbitrary $x \in \overline{\RNp}$, observe that
$$ v_t(x) \to 0, \quad \textup{ as } t \to \infty, \quad \textup{ and } \quad v_{t_1}(x) > v_{t_2}(x), \quad \textup{ for all } 0 < t_1 < t_2.$$
Hence, there exists $t_0 > 0$ such that, for all $t \geq t_0$, 
\begin{equation} \label{choosing t0}
v_t \leq \left(\frac{1}{2p}\right)^{\frac{1}{p-1}} \quad \textup{ in } \overline{\RNp}.
\end{equation}
We fix $t_0 > 0$ such that \eqref{choosing t0} holds and we are going to prove the Step 1 for this $t_0$. To that end, we fix an arbitrary $t \geq t_0 > 0$. First, we are going to prove that $v \geq v_t$ in $\overline{\RNp}$. Since $c > \woo \geq \max_{x \in \R} v_t(x)$, we have that
\begin{equation} \label{v-minus-vt-first-version}
\left\{
\begin{aligned}
 -\Delta (v-v_t) + (v-v_t) & = |v|^{p-1} v-|v_t|^{p-1} v_t, \quad & \textup{ in } \RNp,\\
 v-v_t & > 0, & \quad \textup{ on } \partial \RNp,
\end{aligned}
\right.
\end{equation}
or equivalently
\begin{equation}
\left\{
\begin{aligned}
 -\Delta (v-v_t) + c_t(x)(v-v_t) & = 0, \quad & \textup{ in } \RNp,\\
 v-v_t & > 0, & \quad \textup{ on } \partial \RNp,
\end{aligned}
\right.
\end{equation}
where
\begin{equation}
c_t(x):= \left\{
\begin{aligned}
& 1 - \frac{(v(x))^p - (v_t(x))^p}{v(x)-v_t(x)}, &\quad \textup{ if } v(x) - v_t(x) \neq 0,\\
&1, & \qquad \textup{ if } v(x)-v_t(x) = 0.
\end{aligned}
\right.
\end{equation}
We assume by contradiction that 
\begin{equation}
\big\{x \in \RNp: v(x) < v_t(x) \big\} \neq \emptyset.
\end{equation}
Then, using the mean value theorem and \eqref{choosing t0}, we deduce that, for all $ x \in \{x \in \RNp: v(x) < v_t(x) \}$,
\begin{equation} \label{ct12}
c_t(x) \geq 1-p (v_t(x))^{p-1} \geq \frac{1}{2}.
\end{equation}
Hence, in each connected component $D$ of $\{x \in \RNp: v(x) < v_t(x)\}$ we have that
\begin{equation}
\left\{
\begin{aligned}
 -\Delta (v-v_t) + c_t(x)(v-v_t) & = 0, \quad & \textup{ in } D,\\
 v-v_t & = 0, & \quad \textup{ on } \partial D,
\end{aligned}
\right.
\end{equation}
with $c_t$ satisfying \eqref{ct12}. Then, applying the weak maximum principle \cite[Lemma 2.1]{BeCaNi1997-2}, we obtain that $v \geq v_t$ in $\overline{D}$ which contradicts the fact that $D \subset \{x \in \RNp: v(x) < v_t(x) \}$. Hence, we conclude that $\{x \in \RNp: v(x) < v_t(x) \} = \emptyset$ and so, that $v \geq v_t$ in $\RNp$. Having this at hand and substituting in \eqref{v-minus-vt-first-version}, we deduce that
\begin{equation}
\left\{
\begin{aligned}
 -\Delta (v-v_t) + (v-v_t) & \geq 0, \quad & \textup{ in } \RNp,\\
 v-v_t & > 0, & \quad \textup{ on } \partial \RNp,
\end{aligned}
\right.
\end{equation}
and so, the Step 1 follows from the strong maximum principle and the fact that $t \geq t_0$ is arbitrary.

\medbreak
\noindent \textbf{Step 2: } \textit{ $v > v_t$ in $\overline{\RNp}$ for all $t \in \R$. }
\medbreak
Note that, if we prove that $v \geq v_t$ in $\RNp$ for all $t \in \R$, then the claim follows from the Strong Maximum principle. Also, by the Step 1, we know that 
$$
\big\{\, t \in \R: v \geq v_s \textup{ in } \RNp \textup{ for all } s \geq t\, \big\} \neq \emptyset.
$$
Hence, we can define
\begin{equation} \label{t-star}
t_{\star} := \inf \big\{ \, t \in \R: v \geq v_s \textup{ in } \RNp \textup{ for all } s \geq t\, \big\} \in [-\infty, t_0].
\end{equation}
We argue by contradiction and suppose that $t_{\star} > - \infty$. First note that, by continuity, $v \geq v_{t_{\star}}$ in $\RNp$. Also, $t_{\star} > - \infty$ implies the existence of $M > 0$ such that
\begin{equation} \label{choosing M}
v_{t}(x',x_N)  \leq \left(\frac{1}{2p}\right)^{\frac{1}{p-1}}, \quad \textup{ for all } t \in [t_{\star}-1, t_{\star}],\ x' \in \R^{N-1} \textup{ and } \ x_N \geq M.
\end{equation}
We now consider separately two cases.
\medbreak
\noindent \textit{Case 1: } $\displaystyle \inf_{x \in \R^{N-1} \times [0,M]}\big(v-v_{t_{\star}}\big) =: \delta_M > 0$.
\medbreak
First, taking into account that $\|w_0'\|_{L^{\infty}(\R)} \leq \woo$, we infer that, for all $t \leq t_{\star}$ and $x \in \R^{N-1} \times [0,M]$,
\begin{equation*}
\begin{aligned}
v(x',x_N) - v_t(x',x_N) & = v(x',x_N) - v_{t_{\star}}(x',x_N) + \big(v_{t_{\star}}(x',x_N) - v_t(x',x_N) \big) \\
& \geq \delta_M - \big| w_0(x_N+ t_{\star}) - w_0(x_N+t) \big| \\
& \geq \delta_M - c_p |t_{\star}-t|.
\end{aligned}
\end{equation*}
Hence, there exists $\eta_0 \in (0,1)$ such that, for all $t_{\star} \geq t \geq t_{\star}-\eta_0$, 
\begin{equation} \label{strict-ineq-eta-0}
v(x',x_N) - v_t(x',x_N) > 0, \quad \textup{ for all } x' \in \R^{N-1}\ \textup{ and } \ x_N \in [0,M].
\end{equation}
In particular, if we define $\Sigma_M:= \{ x \in \RN: x_N > M\}$, we have
$$
v - v_t > 0, \quad \textup{ on } \partial \Sigma_M, \quad  \textup{ for all } t \in [t_{\star} - \eta_0, t_{\star}]. 
$$
Next, we are going to prove that, for all $t \in [t_{\star}- \eta_0, t_{\star}]$, it follows $v \geq v_t$ in $\Sigma_M$. To that end, we fix an arbitrary $t \in [t_{\star}- \eta_0, t_{\star}]$. Arguing as in Step 1, we have that
\begin{equation}
\left\{
\begin{aligned}
 -\Delta (v-v_t) + c_t(x)(v-v_t) & = 0, \quad & \textup{ in } \Sigma_M,\\
 v-v_t & > 0, & \quad \textup{ on } \partial \Sigma_M,
\end{aligned}
\right.
\end{equation}
where
\begin{equation}
c_t(x):= \left\{
\begin{aligned}
& 1 - \frac{(v(x))^p - (v_t(x))^p}{v(x)-v_t(x)}, &\quad \textup{ if } v(x) - v_t(x) \neq 0,\\
&1, & \qquad \textup{ if } v(x)-v_t(x) = 0.
\end{aligned}
\right.
\end{equation}
We assume by contradiction that 
\begin{equation}
\big\{x \in \Sigma_M: v(x) < v_t(x) \big\} \neq \emptyset.
\end{equation}
Then, using the mean value theorem and \eqref{choosing M}, we deduce that, for all $ x \in \{x \in \Sigma_M: v(x) < v_t(x) \}$,
\begin{equation} \label{ct13}
c_t(x) \geq 1-p (v_t(x))^{p-1} \geq \frac{1}{2}.
\end{equation}
Hence, in each connected component $D$ of $\{x \in \Sigma_M: v(x) < v_t(x)\}$ we have that
\begin{equation}
\left\{
\begin{aligned}
 -\Delta (v-v_t) + c_t(x)(v-v_t) & = 0, \quad & \textup{ in } D,\\
 v-v_t & = 0, & \quad \textup{ on } \partial D,
\end{aligned}
\right.
\end{equation}
with $c_t$ satisfying \eqref{ct13}. Then, applying the weak maximum principle \cite[Lemma 2.1]{BeCaNi1997-2}, we obtain that $v \geq v_t$ in $\overline{D}$ which contradicts the fact that $D \subset \{x \in \Sigma_M: v(x) < v_t(x) \}$. Hence, we conclude that $\{x \in \Sigma_M: v(x) < v_t(x) \} = \emptyset$ and so, that $v \geq v_t$ in $\Sigma_M$. Taking into account \eqref{strict-ineq-eta-0}, we infer that, for all $\eta \in [0,\eta_0]$, $v \geq v_{t_{\star}-\eta}$ in $\RNp$. This is in contradiction with the definition of $t_{\star}$. Hence, \textit{Case 1} cannot happen. 

\medbreak
\noindent \textit{Case 2:} $\displaystyle \inf_{x \in \R^{N-1} \times [0,M]} \big( v - v_{t_{\star}} \big) = 0$.
\medbreak
In this case there exists a sequence of points $(x^n)_n \subset \R^{N-1} \times [0,M]$ such that
\begin{equation}
v(x^n) - v_{t_{\star}}(x^n) \to 0, \quad \textup{ as } n \to \infty.
\end{equation}
Up to a subsequence, it follows that $x_N^n \to \overline{x}_N$ for some $\overline{x}_N \in [0,M]$. We define then
$$
v^n(x) = v\big(x' + (x^n)', x_N\big)\,, \quad \textup{ for all } n \in \N,
$$
and, for all $n \in \N$, we have $v^n \geq v_{t_{\star}}$ in $\RNp$ and 
\begin{equation*}
\left\{
\begin{aligned}
-\Delta v^n + v^n & = (v^n)^p, \quad & \textup{ in } \RNp,\\
v^n & = c, & \textup{ on } \partial \RNp.
\end{aligned}
\right.
\end{equation*}
Moreover, for all $n \in \N$, it follows that
\begin{equation*}
\left\{
\begin{aligned}
-\Delta (v^n-v_{t_{\star}}) + (v^n-v_{t_{\star}}) & \geq 0, \quad & \textup{ in } \RNp,\\
v^n - v_{t_{\star}} & > 0, & \textup{ on } \partial \RNp,
\end{aligned}
\right.
\end{equation*}
and so, by the Strong Maximum principle, we have that
$$
v^n - v_{t_{\star}} > 0, \quad \textup{ in } \overline{\RNp}, \quad \textup{ for all } n \in \N.
$$
Now, arguing as in \cite[Proof of Theorem 2.1, Step 1]{Fa2003}, we deduce that the sequence $(v^n)_n$ admits a subsequence, still denoted by $(v^n)_n$, converging to a function $\overline{v}$ in $C_{loc}^{2}(\RNp)$. This function $\overline{v}$ still solves 
\begin{equation*}
\left\{
\begin{aligned}
-\Delta \overline{v} + \overline{v}& = \overline{v}^p, \quad & \textup{ in } \RNp,\\
\overline{v} & = c, & \textup{ on } \partial \RNp,
\end{aligned}
\right.
\end{equation*}
and satisfies $\overline{v} \geq v_{t_{\star}}, \textup{ in } \RNp$, and \begin{equation} \label{contradiction-case-2}
 \overline{v}(0',\overline{x}_N) = v_{t_{\star}}(0',\overline{x}_N).
\end{equation} 
Note that \eqref{contradiction-case-2} and $\overline{v} = c > c_p \geq v_{t_{\star}}$ on $\partial \RNp$ imply $\overline{x}_N > 0$. Since $\overline{v} \geq v_{t_{\star}}$ in $\RNp$, we have 
\begin{equation*}
\left\{
\begin{aligned}
-\Delta (\overline{v}-v_{t_{\star}}) + (\overline{v} -v_{t_{\star}})& \geq 0, \quad & \textup{ in } \RNp,\\
\overline{v}-v_{t_{\star}} & > 0, & \textup{ on } \partial \RNp,
\end{aligned}
\right.
\end{equation*}
Hence, by the Strong Maximum principle, it follows that $\overline{v} > v_{t_{\star}}$ in $\RNp$ which gives a contradiction with  \eqref{contradiction-case-2}. \textit{Case 2} cannot happen either and hence the Step 2 follows.

\medbreak
\noindent \textbf{Step 3: }\textit{Conclusion.}
\medbreak
Observe that $v > v_t$ in $\RNp$ for all $t \in \R$ implies that $v \geq v_0(0) = \woo$ in $\RNp$. This gives a contradiction with \eqref{cond-infinity} and so the proof is complete.
\end{proof}

\section*{Declarations}
\label{conflict-interest}
\noindent \textbf{Conflict of Interest.} On behalf of all authors, the corresponding author states that there is no conflict of interest.
\smallbreak
\noindent \textbf{Data Availability Statement.} This article has no additional data.

\bibliographystyle{plain}
\bibliography{Bibliography}
\end{document}